\documentclass[11pt,a4paper]{amsart}
\usepackage{graphicx} 
\usepackage{amssymb}
\usepackage{amsthm}
\usepackage{amsmath}
\usepackage{amsxtra}
\usepackage{latexsym}
\usepackage{mathrsfs}
\usepackage[all,cmtip]{xy}
\usepackage[all]{xy}
\usepackage{enumitem}
\usepackage{xcolor}
\usepackage{graphicx}
\usepackage{comment}
\usepackage{mathabx,epsfig}
\usepackage{stmaryrd}
\usepackage{comment}
\usepackage{pst-node}

\usepackage{tikz-cd} 

\usepackage{hyperref}

\newcommand{\RNum}[1]{\uppercase\expandafter{\romannumeral #1\relax}}

\newtheorem{thm}{Theorem}[section]
\newtheorem{lem}[thm]{Lemma}
\newtheorem{conj}[thm]{Conjecture}

\newtheorem{prop}[thm]{Proposition}
\newtheorem{cor}[thm]{Corollary}

\theoremstyle{definition}
\newtheorem{defn}[thm]{Definition}
\newtheorem{rmk}[thm]{Remark}
\newtheorem{example}[thm]{Example}

\newtheorem{assump}[thm]{Assumption}

\numberwithin{equation}{section}

\newcommand\be{\begin{equation}}
	\newcommand\ba{\begin{eqnarray}}
		\newcommand\ee{\end{equation}}
	\newcommand\ea{\end{eqnarray}}

\def\C{{\mathbb C}}
\def\Q{{\mathbb Q}}
\def\R{{\mathbb R}}

\def\P{{\mathbb P}}
\def\A{{\mathbb A}}
\def\N{{\mathbb N}}

\DeclareMathOperator{\Prep}{Prep}

\DeclareMathOperator{\Supp}{Supp}

\DeclareMathOperator{\Orb}{Orb}

\DeclareMathOperator{\lcm}{lcm}

	{\end{list}}
\title[unlikely intersection and polynomial skew products]{Unlikely intersections in families of polynomial skew products}

\author[Chatchai Noytaptim]{Chatchai Noytaptim}\address{Department of Mathematics\\ Burapha University\\ Chon Buri\\ Thailand 20131}\email{ chatchai.noytaptim@gmail.com}
\author{Xiao Zhong}\address{University of Waterloo \\	Department of Pure Mathematics \\	Waterloo, Ontario \\	Canada  N2L 3G1}\email{x48zhong@uwaterloo.ca}
\date{\today}

\begin{document}
\begin{abstract}
Motivated by the study of unlikely intersection in the moduli space of rational maps, we initiate our investigation on algebraic dynamics for families of regular polynomial skew products in this article. Our goals are threefold.
\begin{enumerate}
    \item We classify special loci---which contain a Zariski dense set of postcritically finite points---in the moduli space of quadratic regular polynomial skew products.  More precisely, special loci include families of homogeneous polynomial endomorphisms, families of split endomorphisms, and polynomial endomorphisms of the form $(x^2,y^2+bx)$ up to conjugacy.  As a consequence, we verify a special case of a conjecture proposed by Zhong.
    \item Let $F_t$ be a family of regular polynomial skew products defined over a number field $K$ and let $P_t, Q_t\in K[t]\times K[t]$ be two initial marked points. We introduce a good height $h_{P_t}(t)$ which is built from the theory of adelic line bundles for quasi projective varieties.
    We show that the set of parameters $t_0\in \overline{K}$ for which $P_{t_0}$ and $Q_{t_0}$ are simultaneously $F_{t_0}$-preperiodic is infinite if and only if $h_{P_t}=h_{Q_t}$. 
    \item As an application of $h_{P_t}$, we show that, under some degree conditions of $P_t$, if there is an infinite set of parameters $t_0$ for which the marked point $P_{t_0}$ is preperiodic under $F_{t_0}$, then the Zariski closure of the forward orbit of $P_t$ lives in a proper subvariety of $\mathbb{P}^2$. As a by-product, we conditionally verify a special case of a conjecture of DeMarco--Mavraki which is a relative version of the Dynamical Manin--Mumford Conjecture.
\end{enumerate}
\end{abstract}
 
\maketitle

	\section{Introduction}
	
\subsection{Background and motivation}
	Unlikely intersection in algebraic dynamics was initiated by the seminal work of Baker-DeMarco. Loosely speaking, in a lot of situations a principle of unlikely intersection can be phrased  as ``a variety contains a Zariski dense subset of special points must itself be special".  In\cite{BD11}, Baker-DeMarco showed that---for fixed complex numbers $a$ and $b$---the set of parameters $c\in\mathbb{C}$ for which $a$ and $b$ are simultaneously preperiodic for $x^d+c$ is infinite if and only if $a^d=b^d$.  This question was raised by Zannier and was itself inspired by the groundbreaking work of Masser-Zannier \cite{MZ08, MZ10, MZ12} in arithmetic geometry regarding unlikely intersection of torsion sections of families of elliptic curves \cite[Chapter 3]{Za12}.
	Since the pioneering work of Baker-DeMarco, there have been several extensive studies and progresses in this direction. Ghioca-Hsia-Tucker \cite[Theorem 2.3]{GHT13} later generalizes \cite[Theorem  1.1]{BD11} to  one-parameter  families of polynomials with non-constant marked points. The same team later established a similar result for families of rational maps parameterized by algebraic curves over a number field \cite[Theorem 1.1]{GHT15}. In the same vein,  \cite{GHT16} considered a more general unlikely intersection instance for ($n+1$)-simultaneous preperiodic points under an $n$-parameter family of dynamical systems.\\
	\indent	A rational function $f$ on $\P^1_{\C}$ of degree $d\geq 2$ is said to be  post-critically finite (PCF, for short) if each of its $2d-2$ critical points  has finite forward orbit. It is known that PCF maps form a Zariski dense subset in the moduli space of rational maps (\cite[\S2.3]{BD13},  \cite[Theorem A]{De18},  and \cite[Proposition 6.18]{Si12} for different ingredients). It is well-known that PCF maps are $\overline{\mathbb{Q}}$-rational points in the moduli space by Thurston rigidity's theorem \cite{DH93}. Outside the locus of flexible Latt$\grave{\text{e}}$s family in  moduli space, PCF maps form a set of bounded Weil height \cite[Theorem 1.1]{BIJL14}. In \cite{BD13}, Baker and DeMarco studied the distribution of PCF maps in the moduli space and aimed to give a characterization of subvarieties in the moduli space of degree $d$ polynomials which contain a Zariski dense set of PCF maps. Our PCF maps can be viewed as special points in the moduli space. This phenomenon was partly motivated by the classical Andr$\acute{\text{e}}$-Oort conjecture, characterizing a subvariety of a Shimura variety which contains a Zariski dense set of CM points \cite[Chapter 4]{Za12}.   
		
		In recent years, there has been substantial progress toward the dynamical Andr$\acute{\text{e}}$-Oort (DAO) conjecture. An outstanding progress was made by Ji-Xie \cite{JX23} who proved DAO for curves in the moduli space of rational maps on $\mathbb{P}^1$. Their tools and methods were different from the one employed by Favre-Gauthier \cite{FG22} who remarkably proved DAO for families of polynomials. In addition, results regarding DAO conjecture  were earlier established in special cases, see \cite{BD11, BD13, FG18, GNK16, GNKY17, GY18}.\\

        Recently, DeMarco and Mavraki \cite{DM24} proposed a far-reaching conjecture, which we will refer to as the \emph{DeMarco--Mavraki Conjecture} in this paper. It is a
relative version of the dynamical Manin--Mumford conjecture studying the Manin--Mumford type question in families of dynamical systems, motivated by the
relative Manin--Mumford theorem of Gao and Habegger \cite{GH23}.
Notably, their work shows that the
conjecture unifies many previously known results on one-dimensional
dynamical unlikely intersections. For example, Ji--Xie’s theorem on DAO is shown to be a special case of the 
DeMarco--Mavraki conjecture. The same is true for the work on the 
simultaneous preperiodicity of marked points discussed in the first 
paragraph (see \cite[Theorem~3.5]{DM24}).

It is worth emphasizing that the DeMarco--Mavraki
conjecture remains largely open. In dimensions greater than one, only
very few cases are currently understood. In \cite{MS25}, Mavraki and
Schmidts established a weaker form of a special case when the dynamical
system is defined on $(\P^1)^n$ (see
\cite[Conjecture~1.9 and Theorem~1.8]{MS25}). More recently, in
\cite{Zho25}, the second author proved a weaker version of the
conjecture in the setting of regular polynomial endomorphisms on
$\P^2$, where the endomorphisms are fixed and the subvariety is allowed
to vary in the family.\\

	 In this article, guided by the framework of the DeMarco--Mavraki Conjecture, 
we study families of regular polynomial skew products with the aim of establishing 
results toward the conjecture in this setting. It is natural, as a first step 
beyond one-dimensional dynamics, to consider polynomial skew products. 
On an affine chart, maps in this class have the form $(f(x), g(x,y))$, 
where the first coordinate evolves according to a one-dimensional dynamical 
system while the second coordinate depends on both variables. 
This provides one of the first classes of dynamical systems exhibiting 
genuinely higher-dimensional behavior, and it has attracted considerable 
attention in recent years (see \cite{AB23, ABDPR16, DFR25, JZ23, JZ20, NZ24, UK20}). The set-up we consider is directly motivated by the results of unlikely intersections and PCF polynomials  discussed above in one dimension.

     \subsection{DeMarco--Mavraki Conjecture}

We follow \cite{DM24} to introduce the necessary notations. An algebraic family of endomorphisms of $\P^n$ of degree $d$ is a morphism $$\Phi:  S \times \P^n \to S \times \P^n$$
given by $\Phi(s,z) =(s, f_s(z)) $ where $f_s$ is an endomorphism of $\P^n$ of degree $d$. Let $\mathcal{X} \subseteq S \times \P^n$ denote a closed irreducible subvariety which is flat over a Zariski open subset of $S$. We use $\mathbf{X}$ to denote the generic fiber of $\mathcal{X}$ and let $\mathbf{\Phi} : \mathbf{P}^n \to \mathbf{P}^n$ be the map induced by $\Phi$, viewed as an endomorphism over the function field $\C(S)$.

We say $\mathcal{X}$ is $\Phi$-special if there exists
 a subvariety $\mathbf{Z} \subseteq \mathbf{P}^n$ over the algebraic closure $\C(S)$ containing the generic fiber $\mathbf{X}$, a
 polarizable endomorphism $\mathbf{\Psi} : \mathbf{Z} \to \mathbf{Z}$, and a positive integer $n$ such that the following hold:
\begin{itemize}
    \item $\mathbf{\Phi}^n(\mathbf{Z}) = \mathbf{Z}$;
    \item $\mathbf{\Phi}^n \circ \mathbf{\Psi} = \mathbf{\Psi} \circ \mathbf{\Phi}^n$ on $\mathbf{Z}$; and
    \item $\mathbf{X}$ is preperiodic under $\mathbf{\Psi}$.
\end{itemize}

We denote $r_{\Phi,\mathcal{X}}$ the relative special dimension of $\mathcal{X}$ over $S$. This is given by 
$$ r_{\Phi, \mathcal{X}} \coloneq \min\{\dim_S\mathcal{Y} : \mathcal{X} \subseteq \mathcal{Y} \text{ and } \mathcal{Y} \text{ is $\Phi$-special}\},$$
where $\dim_S \mathcal{Y} = \dim \mathcal{Y} - \dim S$ is the dimension of a generic fiber of $\mathcal{Y}$ over $S$.

With the notations from above, DeMarco and Mavraki proposed the following relative version of the Dynamical Manin--Mumford Conjecture: 
\begin{conj}[DeMarco--Mavraki Conjecture]\label{conj: DM24-conj-1,1}
     Let $\Phi : S \times \P^N \to S \times \P^N$ be an algebraic family of morphisms of degree $>1$, and let $\mathcal{X} \subseteq S \times \P^N$ be a complex, irreducible subvariety which is flat over $S$. The
 following are equivalent:
 \begin{itemize}
     \item $\mathcal{X}$ contains a Zariski dense set of $\Phi$-preperiodic points.
     \item $\hat{T}^{r_{\Phi.\mathcal{X}}}_\Phi \wedge [X] \neq 0 $ for the relative special dimension $r_{\Phi,\mathcal{X}}$.
     \end{itemize}
\end{conj}
 Here $\hat{T}_\Phi$ is the canonical Green current associated to $\Phi$ on $S \times  \P^N$. Note that the implication from non-vanishing of the current to Zariski density of preperiodic points is proved in \cite[Theorem 1.5]{DM24}. So, to resolve the conjecture, one only needs to focus on the other direction.

\subsection{PCF Quadratic Polynomial Skew Products}
Motivated by the study of the distribution of PCF rational functions in the moduli space of rational functions in one dimension, we study the distribution of PCF endomorphisms in the moduli space of regular polynomial skew products. We obtain a detailed description of families of quadratic regular polynomial skew products with a Zariski dense set of PCF endomorphisms contained in it. 

Since every quadratic regular polynomial skew products can be conjugated to the form 
$$ F(x,y) = (x^2 +d, y^2 + ax^2 + bx + c)$$
by an affine linear transformation on $\A^2$, the moduli space of the quadratic regular polynomial endomorphisms can be naturally identified as $\A^4$; i.e.,
$$ \{F(x,y) = (x^2 + d, y^2 + ax^2 + bx + c) : a,b,c,d \in \C\} \cong \A^4_\C.$$

We establish the following theorem regarding the families containing a Zariski dense set of PCF endomorphisms: 

\begin{thm}\label{thm: pcf-subvariety}
    Let $\mathcal{M}$ denote the moduli space of conjugacy classes of degree-$2$ polynomial skew products, where each class admits a representative of the form
\[
F(x,y) = (x^2 + d,\; y^2 +ax^2+ bx + c),
\qquad d, a, b,c \in \C,
\]
so that $\mathcal{M}$ is naturally identified with $\A^4$.
Let $W \subseteq \A^4$ be an irreducible Zariski closed subset of dimension at least $1$.
If $W$ contains a Zariski dense set of post-critically finite (PCF) points, then $W$ lives in the exceptional locus
\[
\bigl( V(b) \cap V(a) \bigr)\ \cup\ \bigl(V(a) \cap V(d)\cap V(c)\bigr)\ \cup \bigl(V(b) \cap V(c) \cap V(d) \bigr).
\]

\end{thm}

Let $W$ be an affine subvariety of $\A^{4}$. Define a family of quadratic
regular polynomial endomorphisms
\[
\Phi \colon W \times \P^{2} \longrightarrow \P^{2}
\]
by
\[
\Phi(t,[x:y:z]) = F_{t}([x:y:z]),
\]
where $t=(a,b,c,d)\in W\subseteq \A^{4}$ and
\[
F_{t}([x:y:z]) = \bigl[x^{2}+dz^{2} : y^{2}+ax^{2}+bxz+cz^{2} : z^{2}\bigr].
\]
Let
\[
\mathcal{C} \coloneq V(x)\cup V(y)\subseteq W\times \P^{2}
\]
denote the family of critical components of $\Phi$.

From the perspective of the dynamics of families of regular polynomial
endomorphisms, Theorem~\ref{thm: pcf-subvariety} shows that if there
exists a Zariski dense set of parameters $t_{0}\in W$ such that
$\mathcal{C}_{t_{0}}$ is preperiodic under $\Phi_{t_{0}}$, then $W$
must be special, in the sense that
\[
W \subseteq \bigl( V(b) \cap V(a) \bigr)\ \cup\ \bigl(V(a) \cap V(d)\cap V(c)\bigr)\ \cup \bigl(V(b) \cap V(c) \cap V(d) \bigr).
\]

In \cite{Zho25}, the author studied in depth the dynamics of families of
curves under regular polynomial endomorphisms. In particular, in the
case where the fibers of $\Phi$ are constant and equal to a fixed
regular polynomial endomorphism $F$, if a family $\mathcal{X}$ contains
a Zariski dense set of periodic curves, then it has been shown that the family itself must be
periodic: there exists $m\in \N$ such that $F^{m}(C)$ again belongs to
the family for every curve $C$ in $\mathcal{X}$
\cite[Theorem~1.3]{Zho25}.

In the more general setting, where the family $\Phi$ is not assumed to
be constant, one is led to the following conjecture, which is implied
by Conjecture~\ref{conj: DM24-conj-1,1} (see \cite[Lemma~5.5]{Zho25}).

\begin{conj}{\cite[Conjecture 5.4]{Zho25}}\label{conj-const-1.2}
Let $S$ be a smooth and irreducible quasi-projective variety defined over $\C$, 
and let $K$ be a positive integer. 
Let $\Phi : S \times \P^K \to S \times \P^K$ be a family of endomorphisms such that for every point $(s, p) \in S \times \P^K$,
\[
\Phi(s,p) = (s, F_s(p)),
\]
where $F_s$ is an endomorphism of $\P^K$ of degree greater than $1$. 
Let $\mathcal{X} \subseteq S \times \P^K$ be an irreducible subvariety that projects dominantly onto $S$, 
is flat over $S$, and such that $\mathcal{X}_s$ is a preperiodic subvariety under $F_s$ for a Zariski dense set of $s \in S$. 
Then for any positive integer $N$, we have
\[
\hat{T}^{\,r_{\Phi^{\times N}, \mathcal{X}^N}}_{\Phi^{\times N}} \wedge [\mathcal{X}^N] \neq 0.
\]
\end{conj}

As a corollary, Theorem \ref{thm: pcf-subvariety} implies the following special cases of Conjecture \ref{conj-const-1.2}.

\begin{cor}\label{cor: pcf-subv-imply-conj}
  Let $S$ be a smooth, irreducible quasi-projective variety defined over
$\C$. Let
\[
\Phi \colon S \times (\P^{2})^{2} \longrightarrow S \times (\P^{2})^{2}
\]
be a family of endomorphisms such that, for every
$(s,p_{1},p_{2}) \in S \times (\P^{2})^{2}$,
\[
\Phi(s,p_{1},p_{2}) = \bigl(s, F_{s}(p_{1}), F_{s}(p_{2})\bigr),
\]
where for each $s \in S$, the map $F_{s}$ is a regular quadratic
polynomial skew product.

Let
\[
\mathcal{C} \coloneq \mathcal{C}_{x} \times_{S} \mathcal{C}_{y}
\;\subseteq\; S \times (\P^{2})^{2}
\]
be the irreducible subvariety such that, for each $s \in S$, the fibers
$\mathcal{C}_{x,s}$ and $\mathcal{C}_{y,s}$ are the two irreducible
components of the critical locus of $F_{s}$. Suppose that
$\mathcal{C}_{s}$ is preperiodic under $\Phi_{s}$ for a Zariski dense
set of parameters $s \in S$.

Then, for any positive integer $N$, we have
\[
\hat{T}^{\, r_{\Phi^{\times N}, \mathcal{X}^{N}}}_{\Phi^{\times N}}
\;\wedge\; [\mathcal{C}^{N}] \;\neq\; 0.
\]
\end{cor}
 \begin{rmk}
    It is not difficult to see that, via the Segre embedding, we may view the 
dynamics as being defined on $\P^K$ for some suitable $K \in \N$. 
Thus the corollary reduces to the setting of Conjecture~\ref{conj-const-1.2}.
 \end{rmk}

The key argument follows the same strategy as in
\cite[Theorem~1.3]{Zho25}. One verifies that there exists a family of
curves that contains the orbit $\Orb_{F_s}(\mathcal{C}_s)$ for all
$s \in S$ and that is invariant under $F_s$ for every $s \in S$. Then our dynamical pair $(F_s, \mathcal{C}_s)$ induces a dynamical pair $(f_s,p_s)$ on this invariant family where each point $p_s$ represents a curve, $\mathcal{C}_s$, and $f_s(p_s)$ represents $F_s(\mathcal{C}_s)$. We show that the condition \[
\hat{T}^{\, r_{\Phi^{\times N}, \mathcal{X}^{N}}}_{\Phi^{\times N}}
\;\wedge\; [\mathcal{C}^{N}] \;\neq\; 0.
\] is reduced to a similar current condition on $(f_s, p_s)$ which is either one-dimension or a marked pair on $\P^1 \times \P^1$.

To the best of our knowledge, this provides the first known instance,
in essentially dimension greater than one dynamics, of Conjecture~\ref{conj-const-1.2} in
which the endomorphisms are allowed to vary in the family and are not
restricted to split morphisms.

\subsection{Preperiodic points associated to  a one-parameter family of polynomial skew products}

Let $M$ be any complex manifold. Following Astorg-Bianchi \cite{AB22, AB23}, we consider a collection $(F_{t})_{t\in M}$ of holomorphic family of polynomial skew products of degree $d\geq 2$. In other words, a holomorphic map  $F:M\times \mathbb{C}^2\rightarrow\mathbb{C}^2$ such that $F_t:=F(t,\cdot)$ is a polynomial skew product of degree $d$ for all $t\in M$. 

In this subsection, we particularly interested in a family of polynomial skew products of $\mathbb{C}^2$, extendable to $\mathbb{P}^2$, of the form 
$$F_t(x,y):=F(t,x,y)=(f_t(x),g_t(x,y))$$ where $f_t:\mathbb{C}\rightarrow\mathbb{C}$ is a degree $d$ polynomial and $g_t(x,\cdot)$ is also a degree $d$ polynomial for every $x\in\mathbb{C}$ and for all $t\in M$. We refer the reader to subsection \ref{sec:setupactivitybifurcation} for activity and bifurcation discussion pertaining to $F_t$.

In \cite{HK18}, Hsia-Kawaguchi  asserted that---under some natural assumptions on $P_t, Q_t\in \A^2({K[t]})$ over a number field $K$---the set of parameters $t\in \overline{K}$ for which both $P_t$ and $Q_t$ are periodic under the action of one-parameter families H$\acute{\text{e}}$non maps $H_t(x,y)=(y+x^2+t,x)$ is infinite if $P_t$ and $Q_t$ are ``related dynamically" (see \cite[Theorem G]{HK18} more precise description).	 Inspired by the work of Hsia-Kawaguchi, we prove an analogy for one-parameter families of polynomial skew products defined over a number field. Let us consider a one-parameter family of polynomial skew product
$F_t: \A^2\rightarrow\A^2$ parameterized by $t\in \overline{K}$ of degree $d\geq2$:
$$F_t(x,y)=(f_t(x),g_t(x,y))$$ where $f(x)=x^d+O(x^{d-1})\in K[x,t]$ and $g_t(x,y)=y^d+O(y^{d-1})\in K[x,y,t]$. 
Given an initial point $P_t:=(a(t),b(t))\in \mathbb{A}^2(K[t])$ satisfying a natural condition described in subsection \ref{sec:setupactivitybifurcation}.

Using the theory of adelic line bundle for quasi projective varieties recently developed by Yuan-Zhang \cite[Chapter 6]{YZ26}, we obtain a height function $h_{P_t}: \overline{K}\rightarrow\mathbb{R}$. This is a good height function for our context in the sense that it detects parameters $t$ such that $P_t$ is $F_t$-preperiodic.
Denote by $$\text{Prep}(P_t):=\{t\in \overline{K}: f_t^m(P_t)=f_t^n(P_t)\,\,\,0\leq m<n\}$$ the set of all algebraic parameters $t$ so that the forward orbit of $P_t$ under $F_t$ is finite.  
Interestingly, the set  $\text{Prep}(P_t)$ is not always infinite (cf. Remark \ref{rem:exampleofsetprep}).
In order to state our result, let $P_t, Q_t\in \mathbb{A}^2({K[t]})$ be two initial points so that  both $\text{Prep}(P_t)$ and $\text{Prep}(Q_t)$ are infinite. Furthermore, $P_t$ and $Q_t$ satisfy an active assumption in subsection \ref{sec:setupactivitybifurcation}. 
Our main result is stated as follows:
\begin{thm}\label{thm:unlikelytwoinitialpoints} Let $P_t$ and $Q_t$ be as above. Then the following are equivalent:
\begin{enumerate}
\item[(a)] $\text{Prep}(P_t)\cap\text{Prep}(Q_t)$ is infinite;
\item[(b)] $\text{Prep}(P_t)=\text{Prep}(Q_t)$;
\item[(c)] $h_{P_t}=h_{Q_t}$.
\end{enumerate}
\end{thm}
Note that Theorem \ref{thm:unlikelytwoinitialpoints} provides similar statement to that of Hsia-Kawaguchi. However, there are significant details to overcome in our setting. For instance, one needs to understand under which circumstance the set $\text{Prep}(P_t)$ is infinite (cf. \S \ref{sec:prependo}).  In addition, our proof heavily relies on arithmetic equidistribution of Yuan and Zhang. In order to apply it, we need to verify non-degeneracy condition and smallness. For more details, we refer the reader to subsection \ref{sec:unlikelyinter}.

\subsection{Unlikely Intersection and a special case of Conjecture \ref{conj: DM24-conj-1,1}}

With the height function for the family of regular polynomial skew products discussed above, we show that, under some degree restrictions, if a family of regular polynomial skew products with a marked points consisting of a Zariski dense set of preperiodic points, then the marked point is special: 
\begin{thm}\label{thm:mainpolyskewinsubvar}
  
 Let $K$ be a number field and let 
\[
F_t(x,y)=(f_t(x),g_t(x,y))
\]
be a one-parameter family of regular polynomial skew products of degree $d\ge2$, 
with $f_t(x)$ and $g_t(x,y)$ monic polynomials in $K[t][x]$ and $K[t][x,y]$, 
respectively. 

Given a point $P_t := (a(t),b(t)) \in K[t]\times K[t]$ such that
\begin{enumerate}
\item 
$
\deg(a(t)) > \deg_t(f_t);
$

\item $\deg(b(t))$ is positive and
\[
\deg(b(t)) \ge 
\left(
\frac{\operatorname{lcm}(\deg(b(t)),\,\deg(a(t)))}{\deg(b(t))} + 1
\right)
\bigl(\deg(a(t)) + \deg_t(g_t)\bigr).
\]
\end{enumerate}
   Suppose that there are infinitely many $t_0\in \overline{K}$ for which $P_{t_0}$ is preperiodic under the action of $F_{t_0}$. Then the Zariski-closure forward orbit $\overline{\text{Orb}_{F_t}(P_t)}$ is contained in a proper subvariety of $\mathbb{P}^2_{\overline{K}(t)}$. 

\end{thm}

\begin{rmk}
    The first condition, $\deg(a(t)) > \deg_t(f_t)$, can always be achieved by 
replacing $P_t$ with a point in its forward orbit, unless $a(t)$ is 
preperiodic under $f_t$ or the pair $(a(t),f_t)$ is isotrivial. 
This follows by considering the height of the iterates $a_n(t)$ over the 
function field $K(t)$, which coincides with $\deg(a_n(t))$ for each 
$n \in \mathbb{N}$. If the height of $a_n(t)$ does not grow as 
$n \to \infty$, then the canonical height of $a(t)$ with respect to $f_t$ 
is zero. Consequently, $a(t)$ is either preperiodic or the pair 
$(a(t),f_t)$ is isotrivial.

In either of these cases, however, the theorem becomes trivial. 
Indeed, if one of these two cases holds and there exists $t_0$ such that $a(t_0)$ is preperiodic under 
$f_{t_0}$, then $a(t)$ is preperiodic under $f_t$ as a point in $\P^2_{K(t)}$. 
Hence the Zariski closure of the orbit is automatically contained in a 
proper subvariety of $\mathbb{P}^2$.

Therefore, the essential requirement for the theorem is that, after 
replacing $P_t$ by a sufficiently large iterate in its orbit, 
condition~(2) holds.
\end{rmk}
We also present the following example—distinct from the split and homogeneous
cases—to demonstrate that Theorem~\ref{thm:mainpolyskewinsubvar} is non-vacuous.

\begin{rmk}
Consider the regular polynomial skew product
\[
F_t(x,y) = \bigl(x^{11},\, y^{11} + t y^2 - t x^{11}\bigr),
\]
defined over $\Q[t]$, and the point
\[
P_t = (a(t), b(t)) = (t^2, t^{11}).
\]

Then $P_t$ satisfies the hypotheses of
Theorem~\ref{thm:mainpolyskewinsubvar}, since
\[
\deg_t(b) = 11 > 9
= \left( \frac{\lcm(\deg_t(a), \deg_t(b))}{\deg_t(b)} + 1 \right)
(\deg_t(a) + 1).
\]
Consequently, if there exist infinitely many $t_0 \in \overline{\Q}$ such that
$P_{t_0}$ is preperiodic under $F_{t_0}$, then the orbit closure
$\overline{\Orb_{F_t}(P_t)}$ is a proper $F_t$-invariant subvariety of $\A^2$.

In fact, in this example we have
\[
\overline{\Orb_{F_t}(P_t)} = V(y^2 - x^{11}),
\]
which is invariant under $F_t$. Moreover, for any root of unity $t_0$, the point
$P_{t_0}$ is preperiodic under $F_{t_0}$.
\end{rmk}
\begin{rmk} Note that given $F_t$ is a flat family of regular polynomial skew products 
(i.e., for every $t_0 \in \overline{K}$, the map $F_{t_0}$ extends to an 
endomorphism of $\P^2$ of fixed degree $d>1$), we may assume after conjugation 
that $F_t$ is monic; that is, both $f_t(x)$ and $g_t(x,y) \in K[t,x][y]$ are monic.

Indeed, consider a polynomial skew product of the form
\[
H_t(x,y)=\bigl(k_1x^d+O(x^{d-1}),\; k_2y^d+O_{x,t}(y^{d-1})\bigr).
\]
Then $H_t$ is dynamically conjugate to a monic map via the linear change of 
coordinates
\[
\varphi(x,y)=\bigl(k_1^{1/(d-1)}x,\; k_2^{1/(d-1)}y\bigr).
\]
In other words, the conjugated map $\varphi \circ H_t \circ \varphi^{-1}(x,y)$ 
is monic.\end{rmk}
As a corollary, Theorem \ref{thm:mainpolyskewinsubvar} also proves a special case of Conjecture \ref{conj: DM24-conj-1,1} under the degree constriction:
\begin{cor}\label{cor: main-skew-imply-conj}
   Let $K$ be a number field. Let
\[
\Phi \colon \A^{1} \times \P^{2} \longrightarrow \A^{1} \times \P^{2}
\]
be an algebraic family of regular polynomial endomorphisms of degree
$d>1$, given by
\[
\Phi\bigl(t,[x:y:z]\bigr)
=
\Bigl(t,\,
\bigl[z^{d} f_t(x/z) : z^{d} g_{t}(x/z,y/z) : z^{d}\bigr]\Bigr),
\]
where $f_t(x) \in K[t][x]$ and $g_{t}(x,y) \in K[t][x,y]$.

Let
\[
\mathcal{X} \coloneq \bigl(t, [a(t): b(t):1]\bigr)
\subseteq \A^{1} \times \P^{2}
\]
be a marked point over $\A^{1}$, where $a(t), b(t) \in K[t]$ such that 
$$ \deg(a(t)) > \deg_t(f_t),$$
$$\deg b(t) \;\ge\;
\left(
\frac{\operatorname{lcm}\bigl(\deg b(t), \deg a(t)\bigr)}{\deg b(t)} + 1
\right)
\bigl(\deg a(t) + \deg_t g_t\bigr). $$
Suppose that there exist infinitely many parameters
$t_{0} \in \A^{1}_{\overline{K}}$ such that the specialization
\[
\mathcal{X}_{t_{0}} = \bigl[a(t_{0}): b(t_{0}) : 1\bigr]
\]
is preperiodic under the specialized map
\[
\Phi_{t_{0}}\bigl([x:y:z]\bigr)
=
\bigl[z^{d} f_{t_{0}}(x/z) : z^{d} g_{t_{0}}(x/z,y/z) : z^{d}\bigr].
\]
Then
\[
\hat{T}^{\, r_{\Phi,\mathcal{X}}}_{\Phi} \wedge [\mathcal{X}] \neq 0,
\]
where $r_{\Phi,\mathcal{X}}$ denotes the relative special dimension
associated to $\Phi$ and $\mathcal{X}$.
\end{cor}
\subsection{Outline of the paper}
	In Section \ref{sec:PCF}, we study the distribution of post-critically finite quadratic polynomial skew products in the moduli space and prove Theorem \ref{thm: pcf-subvariety}. The main idea is to reduce the problem to special point classifications for split quadratic polynomial families, and then apply a careful case-by-case analysis of the possible exceptional loci. We also demonstrate how Theorem \ref{thm: pcf-subvariety} implies the case of \cite[Conjecture 5.4]{Zho25}. The main argument is to identify a fiberation where the fiber structures are preserved by $F_s$ for all $s$ in the parameter space. Then the analysis on currents are reduced to one-dimension or split maps cases by projecting to the base of the fiberation. In Section \ref{sec:unlikelypolyskew}, we construct the height function associated to a marked point in a one-parameter family of regular polynomial skew products using adelic line bundles on quasi-projective varieties, and we prove Theorem \ref{thm:unlikelytwoinitialpoints} by combining this construction with arithmetic equidistribution of small points. In Section \ref{sec:proper-sub}, we apply the height functions and the vertical Böttcher coordinate formalism to show that, under suitable degree assumptions, the forward orbit of a marked point is forced to lie in a proper algebraic subvariety if the parameter space contains a Zariski dense set of preperiodic parameters, which yields Theorem \ref{thm:mainpolyskewinsubvar}. We conclude by explaining how these results fit into the framework of the DeMarco--Mavraki conjecture.
\section{Special subvarieties in $\mathcal{M}_2$ : postcritically finite maps}\label{sec:PCF}
	In this section, we denote $f_k(x) \coloneq x^2 + k$ for every $k \in \C$.
    \begin{defn}
Let $W \subseteq \A^n_{\C}$ be a subvariety, and let
$p = (p_1,\dots,p_n) \in W$.
We say that $p$ is a \emph{special point} of $W$ if, for every
$i \in \{1,\dots,n\}$, the polynomial $f_{p_i}$ is post-critically finite.
\end{defn}
\begin{defn}
   Let $\mathcal{M}_d$ denote the moduli space of conjugacy classes of degree $d$ 
regular polynomial skew products. A point $x \in \mathcal{M}_d$ is said to be 
\emph{post-critically finite} (PCF) if it corresponds to a conjugacy class of 
post-critically finite endomorphisms.

\end{defn}

We will prove Theorem \ref{thm: pcf-subvariety} and, as an application, Corollary \ref{cor: pcf-subv-imply-conj}.

\subsection{Two special cases of Theorem \ref{thm: pcf-subvariety}}

In this subsection, we collect two propositions resolving two special subcases of Theorem \ref{thm: pcf-subvariety} where the Moduli space is of dimension $3$. The proof of both two is quite lengthy, but they follow a similar framework applying \cite[Theorem 1.2]{DM25}.

  \begin{prop}\label{prop: b=0=special-subvariety}
  Let $\mathcal{M}$ denote a subspace of the moduli space of conjugacy classes of degree-$2$ 
polynomial skew products, where each class admits a representative of the form
\[
F(x,y) = (x^2 + d,\; y^2 + ax^2 + c),
\qquad d,c \in \C,~ a \in \C^*,
\]
so that $\mathcal{M}$ is naturally identified with $\C^* \times \A^2$ via the parameters $(a,c,d)$.

        Suppose $W \subseteq \mathcal{M}$ is an irreducible subvariety of positive dimension that contains a Zariski dense set of PCF points. Then $W = V(d) \cap V(c)$. 
    \end{prop}
    \begin{proof}
    
        Let $d_1$ denotes a solution of $x^2 + d =x$ for each $d \in \C$. Then let $\mu : \C^* \times \A^2 \to \C^* \times \A^2$ be given by 
        $$ \mu(a,c,d_1) = ( a,c,d_1 -d^2_1) = (a,c,d),$$
        for any $( a,c,d_1) \in \C^* \times  \A^2$. Let $W' \coloneq \mu^{-1}(W) \subseteq \C^* \times \A^2$. It is enough to show that $W' \subseteq V(d_1-d^2_1) \cap V(c)$. We work with each irreducible component separately and hence we assume from now on that $W'$ is irreducible. 

        Let $\sigma : \C^* \times \A^2 \to \C^* \times \A^3$ be the finite map given by
        $$ \sigma(a,c,d_1) = (a, ad^2_1 + c, a(1-d_1)^2+c, d_1 -d^2_1),$$
        for any $(a,c,d_1) \in \C^* \times \A^2$. 

        Note that if $V(y)$ is preperiodic under $F$. Then, since $d_1$ is a fixed point under $f_d$, we have $$\Orb_{f_{ad^2_1 + c}}(0) \subseteq \pi_y \left(\bigcup_{n \in \N}F^n(V(y)) \cap V(x- d_1) \right),$$
        which is a finite set and hence $f_{ad^2_1 + c}$ is a post-critically finite polynomial. The same argument will also imply that $f_{a(1-d_1)^2 + c}$ is post-critically finite since $1-d_1$ is also a fixed point under $f_d$. Moreover, if $F(x,y)$ is post-critically finite, then $f_d$ and $f_a$ are also post-critically finite polynomials. 
        
        Hence the assumption that there exists a Zariski dense set of points in $(a,c,d) \in W$ such that $F(x,y) = (x^2 + d, y^2 + ax^2 + c)$ is post-critically finite implies that $\overline{\sigma(W')}$ contains a Zariski dense set of special points. 

        Suppose $\dim(W) = 3$. Then we also have $\dim\left(\overline{\sigma(W')}\right) = \dim(W') = 3 $. Since $\overline{\sigma(W')}$ contains a Zariski dense set of special points, \cite[Theorem 1.2]{DM25} implies that 
        $$ \overline{\sigma(W')} \subseteq  \bigcup_{D_1: f_{D_1} \text{ is PCF}}\left( \bigcup^4_{i = 1} V(x_i- D_1)\right) \cup \left(\bigcup_{1 \leq i\neq j\leq 4}V(x_i -x_j)\right),$$
        where $x_1, x_2,x_3,x_4$ are the coordinates of $\C^* \times \A^3$.
        
        This implies that there exists a polynomial $p$ in 
        \begin{align}
           \mathcal{L}\coloneq \{ d_1 -d^2_1 -D_1, ad^2_1 +c -D_1, a(1-d_1)^2 + c -D_1, d_1 -(a+1)d^2_1 -c , \nonumber\\ d_1 - (1-a)d^2_1 - a- c+ 2ad_1, a(2d_1 -1), a- D_1, d-d^2_1 -a, ad^2_1 + c-a, \nonumber\\a(1-d_1)^2 + c-a: D_1 \in \C, \text{ } f_{D_1} \text{ is post-critically finite}\}
        \end{align}
        such that $W' \subseteq V(p)$. This implies that $\dim(W') < 3$, which gives the contradiction.\\

        Suppose $\dim(W) = 2$. Then similarly, we have $\dim(\overline{\sigma(W')}) \leq  2$. Again, by \cite[Theorem 1.2]{DM25}, there exist two distinct polynomials $p_1, p_2 \in \mathcal{L}$ such that 
        $$ W' \subseteq V(p_1) \cap V(p_2).$$
        Note that if one of $p_1$ and $p_2$ is $a(2d_1-1) $, then $W' \subseteq V(a) \cup V(d_1 - 1/2)$. Since $a\neq 0$ by assumption and $d_1 -1/2$ implies $d = d_1 - d^2_1 = 1/4$, which won't satisfy that $f_d$ is a post-critically finite polynomial, we have that $W'$ cannot be a subvareity in $\C^* \times \A^2$ contains a Zariski dense set of special points. This is a contradiction. Thus, neither $p_1$ nor $p_2$ can be $a(2d_1-1)$. So $p_1, p_2 \in \mathcal{L}' \coloneq \mathcal{L} \setminus \{a(2d_1 -1)\}$. 
        
        Now, except $d_1 -d_1^2 - D_1$, for some $D_1$ such that $f_{D_1}$ is post-critically finite, polynomials in $\mathcal{L}'$ are all irreducible and monic in either $a$ or $c$. Hence, if $p_1, p_2$ are not in the above exceptions, we have that $p_1$ and $p_2$ do not have a non-trivial greatest common divisor as polynomials in $\C[a,c,d_1]$. Thus, $\dim(W') \leq 1$, which gives the contradiction. Now, suppose $p_1$ is given by $d_1 - d^2_1 - D_1$ for a $D_1 \in \C$ such that $f_{D_1}$ is post-critically finite. Since $p_2 \in \mathcal{L}'\setminus\{p_1\}$, we have that $V(p_1) \cap V(p_2)$ will obviously be codimensional $2$ in $\C^* \times \A^2$ as $p_2$ is monic in either $a$ or $c$. This implies that $\dim(W') \leq 1$, which is a contradiction.\\

        Now, suppose $\dim(W') = 1$. We first assume that the projection to the fourth coordinate of $\overline{\sigma(W')}$ is dominant. Then we look at $\pi_{1,4}(\overline{\sigma(W)})$, where $\pi_{i,j}$ for some $i\neq j\in\{1,2,3,4\}$ denotes the projection to the $i,j$-th coordinates. Then since $\dim(\pi_{1,4}(\overline{\sigma(W')}))=1$, \cite[Theorem 1.2]{DM25} implies that either 
        $$ a-d_1 +d^2_1 = 0$$
        or
        $$ a = D_1$$
        for some $D_1 \in \C$ such that $f_{D_1}$ is post-critically finite. 

        {\bf Case (1)}: Suppose \begin{equation}\label{eq: b=0-1-case-1}
            a-d_1 +d^2_1 = 0.
        \end{equation} In particular, the projection of $\overline{\sigma(W')}$ to the first coordinate is also dominant. Then $\dim(\pi_{1,2}(\overline{\sigma(W')})) = 1$ and \cite[Theorem 1.2]{DM25} implies that one of the following holds: 
        \begin{enumerate}
            \item $ad_1^2  + c - D_2 = 0$; 
            \item $a - D_2 = 0$;
            \item $a-ad^2_1 - c =0$;
        \end{enumerate}
        where $D_2 \in \C$ such that $f_{D_2}$ is post-critically finite. Note that the second case will contradict that $\overline{\sigma(W')}$ projects dominantly to the first coordinate. 

        {\bf Subcase (i):} Suppose 
        \begin{equation}\label{eq: b=0-1-case-1-1}
            ad^2_1 +c -D_2 = 0,
        \end{equation}
        
        for some $D_2 \in \C$. Then $\dim(\pi_{1,3}(\overline{\sigma(W')})) = 1$ and \cite[Theorem 1.2]{DM25} will again implies that one of 
        $$ a(1 - d_1)^2 + c -D_3 = 0,$$
        and 
        $$ a - a(1-d_1)^2 - c =0$$
        will hold, where $D_3 \in \C$ satisfies that $f_{D_3}$ is post-critically finite. If $a(1-d_1)^2 + c - D_3 = 0$, then plug this back to Equation (\ref{eq: b=0-1-case-1-1}), we have 
        $$ a(1-2d_1)+ D_1 -D_2 = 0.$$
        Then, together with Equation (\ref{eq: b=0-1-case-1}), we obtain that $d_1$ can take only finitely many values. This is a contradiction. 

        If $a - a(1-d_1)^2 - c =0$, then combined with Equation (\ref{eq: b=0-1-case-1-1}), we obtain that 
        $$ 2ad_1 - D_2= 0.$$
        Again, with Equation (\ref{eq: b=0-1-case-1}), we have that $d_1$ can only take finitely many values. 

        {\bf Subcase (ii):} Suppose 
        \begin{equation}\label{eq: b=0-1-case-1-2}
            a-ad^2_1 - c = 0.
        \end{equation}
        Then $\dim(\pi_{1,3}(\overline{\sigma(W')})) = 1$ and \cite[Theorem 1.2]{DM25} again implies that one of 
        $$ a(1 - d_1)^2 + c -D_3 = 0,$$
        and 
        $$ a - a(1-d_1)^2 - c =0$$
        will hold, where $D_3 \in \C$ satisfies that $f_{D_3}$ is post-critically finite. 
        If $a(1-d_1)^2 + c - D_3 = 0$ holds, then plugging this back to Equation (\ref{eq: b=0-1-case-1-2}) implies that 
        $$ 2a - 2ad_1 - D_3 = 0.$$
        Since $a \in \C^*$, this together with Equation (\ref{eq: b=0-1-case-1}) implies that $d_1$ can only take finitely many values. 

        If $a - a(1-d_1)^2 - c =0$, then together with Equation (\ref{eq: b=0-1-case-1-2}), we have $a(1-2d_1)= 0$, which gives that $d_1$ can only take finitely many values since $a \in \C^*$.\\

    {\bf Case (2):} Now, suppose
    \begin{equation}\label{eq: b=0-1-case-2}
        a - D_1 = 0.
    \end{equation}
    Then, similarly, $\dim(\pi_{2,4}(\overline{\sigma(W')})) = 1$ and \cite[Theorem 1.2]{DM25} implies that one of 
    \begin{equation}\label{eq: b=0-1-case-2-1}
        ad^2_1 + c-D_2 = 0,
    \end{equation}
    and 
    \begin{equation}\label{eq: b=0-1-case-2-2}
        ad^2_1 + c -d_1 + d^2_1 = 0,
    \end{equation}
    holds, where $D_2 \in \C$.

    {\bf Subcase (i):} Suppose Equation (\ref{eq: b=0-1-case-2-1}) holds. Then, $\dim(\pi_{2,4}(\overline{\sigma(W')})) = 1$ and \cite[Theorem 1.2]{DM25} implies that 
    one of $$a(1-d_1)^2 + c - D_3=0,$$
    $$ a(1-d_1)^2 + c -d_1 + d^2_1 = 0$$
    holds, where $D_3 \in \C$. 
    If $a(1-d_1)^2 + c - D_3$ for some $D_3 \in \C$, then together with Equation (\ref{eq: b=0-1-case-2-1}) we have 
    $$ a(1-2d_1) + D_2 -D_3 = 0.$$
    This, together with Equation (\ref{eq: b=0-1-case-2}) implies that $d_1$ can only take finitely many values, which is a contradiction.

    If $a(1-d_1)^2 + c- d_1 + d^2_1= 0$, then plugging it back to Equation (\ref{eq: b=0-1-case-2-1}) gives that 
    $$ a(1-2d_1) + ad^2_1 + D_2 -d_1 = 0.$$
    Again, with Equation (\ref{eq: b=0-1-case-2}), we have that $d_1$ can only take finitely many values.

    {\bf Subcase (ii):} Suppose 
    \begin{equation}\label{eq: b=0-1-case-2-2}
        ad^2_1 + c -d_1 + d^2_1 = 0.
    \end{equation}
    Then, again, with $\dim(\pi_{3,4}(\overline{\sigma(W')})) = 1$, we have 
    one of $$a(1-d_1)^2 + c - D_3=0,$$
    $$ a(1-d_1)^2 + c -d_1 + d^2_1 = 0$$
    holds, where $D_3 \in \C$. 
    If $a(1-d_1)^2 + c - D_3$ for some $D_3 \in \C$, then together with Equation (\ref{eq: b=0-1-case-2-2}) we have
    $$ a(1-2d_1) -d^2_1 + d_1 -D_3 = 0,$$
    which together with Equation (\ref{eq: b=0-1-case-2}) implies that $d_1$ can only take finitely many values. 

    If $a(1-d_1)^2 +c -d_1 + d^2_1 = 0$ holds, then similarly together with Equation (\ref{eq: b=0-1-case-2-2}) we have
    $$ a(1-d_1)^2 -ad^2_1 = a(1-2d_1) = 0.$$
    Since $a = D_1 \in \C^*$, we have 
    that $d_1$ can only take finitely many values. Both of the cases contradict our assumption that the projection of $\overline{\sigma(W')}$ to the forth coordinate is dominant.\\

    Now, we are left with the case that $d_1$ takes only finitely many values. Since $\overline{\sigma(W')}$ is irreducible, we assume that $d_1$ takes a fixed value for all points in $\overline{\sigma(W')}$. Since $\dim(\overline{\sigma(W')}) =1$, we have 
    $$ \dim (\pi_{1,2}(\overline{\sigma(W')})) = 1. $$
    Then, \cite[Theorem 1.2]{DM25} implies that 
    one of the following holds: 
    \begin{enumerate}
        \item $a - D_1 = 0$;
        \item $ad^2_1 + c -D_1 = 0$;
        \item $ad^2_1 + c - a= 0$;
    \end{enumerate}
    where $D_1 \in \C$ such that $f_{D_1}$ is post-critically finite. 
    
    {\bf Case (1):} If $a = D_1$, then the assumption that $\dim(W') = 1$ will imply that 
    $$ \dim(\pi_{2,3}(\overline{\sigma(W')})) = 1,$$ and applying \cite[Theorem 1.2]{DM25} again, we will get that one of the following holds: 
    \begin{enumerate}
        \item $ad^2_1 + c -D_2 = 0$;
        \item $a(1-d_1)^2 + c - D_2 = 0$;
        \item $a(1-2d_1)= 0$;
    \end{enumerate}
    where $D_2 \in \C$ such that $f_{D_2}$ is post-critically finite. 
    Since $a \in \C^*$ and $d_1 \neq 1/2$, as then $f_{d_1 - d^2} = f_{1/4}$ is not post-critically finite, we have $a(1 - 2d_1) \neq 0$. If one of the first two equations hold, then obviously that $c$ is also determined by $a= D_1$ and $d_1$, which contradicts that $\dim(W') = 1$. \\

    {\bf Case (2):} Suppose 
    \begin{equation}\label{eq: b=0-1-case-3}
        ad^2_1 + c -D_1 = 0.
    \end{equation}
    If $d_1 = 0$, then $c - D_1$. In this case, $\dim(\pi_{1,3}(\overline{\sigma(W')})) = 1$ and we have one of $a = D_2$, $a + c = D_2$ and $c = 0$ must hold by \cite[Theorem 1.2]{DM25}. Thus, either $a$ also takes only finitely many values, contradicting the dimension assumption, or $c = 0$ as well. Note that the second case that $c =0$, $d_1 = 0$ satisfies our conclusion.
    
     On the other hand, if $d_1 \neq 0$, we have $\dim(\pi_{1,3}(\overline{\sigma(W')})) = 1$, since if $a$ takes only finitely many values, then so is $c$, which will contradict the assumption that $\dim(W') = 1$. 
     Now, applying \cite[Theorem 1.2]{DM25} to $\pi_{1,3}(\overline{\sigma(W')})$, we have one of the following holds:
    \begin{enumerate}
        \item $a - D_2 = 0$;
        \item $a(1 -d_1)^2 + c - D_2 = 0$;
        \item $c + a(1-d_1)^2 - a=0$;
    \end{enumerate}
    where $D_2 \in \C$ such that $f_{D_2}$ is post-critically finite. 

    If $a = D_2$, then we combine with Equation (\ref{eq: b=0-1-case-3}) to obtain that $c$ is determined by $a$ and $d_1$ and can only take finitely many values. This contradicts that $\dim(W') = 1$. 

    If $a(1-d_1)^2 + c -D_2 = 0$, then 
    together with Equation (\ref{eq: b=0-1-case-3}), we obtain that 
    $$ D_2 - D_1 - a(1-2d_1) = 0.$$
    Note that $d_1 \neq 1/2$ as $f_{1/4}$ is not post-critically finite. Thus, $a = (D_2 -D_1)/(1 -2d_1)$ and again $c$ is determined by $a$ and $d_1$ which can only take finitely many values. This again contradicts the dimension assumption. 

    If $c + a(1-d_1)^2 -a = 0$ then, combined with Equation (\ref{eq: b=0-1-case-3}), we obtain that 
    $$ 2ad_1 - D_1 = 0.$$
    Since we assumed $d_1 \neq 0$ in this case, we have $a = D_2/(2d_1)$ and hence $c$ also takes only finitely many values, contradicting that $\dim(W') = 1$. \\

    {\bf Case (3):} Lastly, suppose
    \begin{equation}\label{eq: b=0-1-case-3-3}
        ad^2 + c -a = 0.
    \end{equation}
    Note that if $d_1 \in \{1,-1\}$, then we have $c = 0$ and $d = d_1 - d_1^2 = 0 $ or $2$. In the case that $d_1 = -1$, $\dim(\pi_{1,3}(\overline{\sigma(W')})) = 1$ and thus, applying \cite[Theorem 1.2]{DM25}, one of the following holds:
    \begin{enumerate}
        \item $a - D_2 =0$;
        \item $a(1-d_1)^2 + c - D_2 = 0$;
        \item $c + a(1 -d_1)^2 - a = 0$;
    \end{enumerate}
    where $D_2 \in \C$ such that $f_{D_2}$ is post-critically finite. In all these cases, we see that $a$ can only take finitely many values, contradicting the dimension assumption.

    On the other hand, suppose $d_1 \not\in \{1,-1\}$. Again, we apply \cite[Theorem 1.2]{DM25} to $\pi_{1,3}(\overline{\sigma(W')})$ and obtain the three cases above. If $a - D_2 = 0$ holds, then again $c = a - ad^2_1$ is determined by $a$ and $d_1$. This gives a contradiction to $\dim(W') = 1$. 

    If $a(1-d_1)^2 + c - D_2 = 0$ holds, then Equation (\ref{eq: b=0-1-case-3-3}) implies that 
    $$ a(1-2d_1) + a - D_2 = 0.$$
    Since $d_1 \neq 1$, we have $a = D_2/(2 - 2d_1)$ and again $c$ is determined by $a$ and $d_1$, which gives a contradiction.

    Finally, if $c + a(1 - d_1)^2 -a = 0$, then together with Equation (\ref{eq: b=0-1-case-3-3}), we have $a(1 -2d_1) = 0$. Since $a \in \C^*$ and $d_1 \neq 1/2$ we again conclude that this case is impossible. 
     \end{proof}

    \begin{prop}\label{prop: pcf-a=0}
Let $\mathcal{M}$ denote a subspace of the moduli space of conjugacy classes of degree-$2$ 
polynomial skew products, where each class admits a representative of the form
\[
F(x,y) = (x^2 + d,\; y^2 + bx + c),
\qquad d,b,c \in \C,
\]
so that $\mathcal{M}$ is naturally identified with $\A^3$ via the parameters $(d,b,c)$.

Let $W \subseteq \mathcal{M}$ be an irreducible Zariski closed subset of 
dimension at least $1$. If $W$ contains a Zariski dense set of PCF points, then $W$ is contained in the exceptional locus
\[
V(b)\ \cup\ \bigl(V(d)\cap V(c)\bigr).
\]

\end{prop}

\begin{proof}
We suppose $W$ is not in the exceptional locus and show that $W$ doesn't contain a Zariski dense set of PCF points.

Suppose $F(x,y) \in \mathcal{M}$ is PCF, then $V(y)$ is preperiodic under $F(x,y)$.
Therefore, there exists a polynomial $P(x,y)$ such that $$F^{nk + l}(V(y)) = V(P(x,y))$$ for some positive integer $k,l$ and any positive integer $n$. Let $P_m(x,y)$ be the largest weighted degree terms in $P(x,y)$ with degree of $x$ weighted by $2$ and degree of $y$ weighted $1$, where $m$ is the weighted degree of $P(x,y)$. 
Note that 
$$ F^n(x,0) = \left(x^{2^n} + o(x^{2^n}), f^n_b(0)x^{2^{n-1}}+o(x^{2^{n-1}})\right).$$
Then, we have 
$$ P\left(F^{nk+l}_1(x,0),F^{nk+l}_2(x,0)\right) = 0,$$
where $F^l_1$ and $F^l_2$ denote the two coordinates of $F^l$ respectively for each $l \in \N$,
implies 
$$ P_m\left(F^{nk+l}_1(x,0),F^{nk+l}_2(x,0)\right) = o\left(x^{m2^{nk+l-1}}\right).$$
Note that, for any constant $\tau \in \C$ and $n\in \N$, the leading term of $P_m\left(x^{2^n},\tau x^{2^{n-1}}\right)$ is 
$$ a_m (\tau)x^{m2^{n-1}}$$
where $a_m$ is a non-trivial polynomial. Therefore $V(a_{m})$ is a finite set and $f^{nk+l}_b(0)$ is in $V(a_m)$ for every $n \in \N$. Thus, $f_b$ is PCF.\\

Now, we suppose $W = \A^3$. We want to show that the set of PCF points in $W$ is not Zariski dense. Let $d_1$ and $d_2$ be two fixed point of $x^2+d$ and notice that $F(x,y)$ is PCF implies $x^2 +d$ is PCF and in particular $d_1 \neq d_2$ as $x^2 +1/4$ is not PCF. We have $F^n(d_1,0)$ is contained in $F^n(V(y))$ which is inside the set of non-vertical subvarieties as $x^2 +d $ is a finite map. Since $V(y)$ is preperiodic if $F$ is PCF, we have $(d_1,0)$ and $(d_2,0)$ has finite forward orbit under $F$ if $F$ is PCF and thus,
since $$F(d_i, y) = (d_i, y^2 + bd_i + c),$$
where $i \in \{1,2\}$,
$f_{d_1b + c}$ and $f_{d_2b +c}$ are PCF maps. 

Since we have the relation $d = d_1 - d^2_1$ and $d_2 = 1 -d_1$. We can consider a lift of $W$ by the finite map
$$ \pi (d_1, b, c) = (d_1 - d^2_1, b ,c).$$
We call a point in $\pi^{-1}(W)$ a PCF point if its image in $W$ is a PCF point. Note that $W$ contains a Zariski dense set of PCF points if and only if $\pi^{-1}(W)$ contains a Zariski dense set of PCF points. Therefore, it is enough to show that if $\pi^{-1}(W)$ is not in $V(b)\ \cup\ \bigl(V(d_1 - d^2_1)\cap V(c)\bigr)$ then $\pi^{-1}(W)$ doesn't contain a Zariski dense set of PCF points. To ease the notation, we denote $\pi^{-1}(W)$ as $W$ from now on.\\

Now, we consider the embedding of $\A^3 \to \A^4$ by the following 
$$ \sigma : (d_1,b,c) \to (d_1 - d^2_1 , d_1b + c, (1 -d_1) b + c,b).$$
We have by the above argument that if $F$ is PCF then $\sigma((d,b,c))$ is a special point in the sense that $f_{\sigma(d,b,c)_i}$'s are PCF for each $i \in \{1,2,3,4\}$. Now, $W$ contains a Zariski dense set of PCF points implies that there is a Zariski dense set of special points in $\overline{\sigma(W)}$. Since $\overline{\sigma(W)}$ is of dimension $3$, we have by \cite[Theorem 1.2]{DM25}, $\overline{\sigma(W)}$ is given by the vanishing set of a polynomial in $$\mathcal{C} = \{ x_i - D, x_i - x_j : i \neq j \in \{1,2,3,4\}, D \in \C \text{ s.t. } f_D \text{ is PCF} \}.$$  
This implies that there exists a polynomial $p$ in 
        \begin{align}
           \mathcal{L} \coloneq \{ d_1-d_1^2 -D, b-D,d_1b+c -D, (1-d_1)b + c - D, d_1-d_1^2 -b, \nonumber\\ d_1(1-b)-d^2_1 -c,  d_1(1-b) -d_1^2 -b -c, (1 -d_1)b -c, d_1b -c, (2d_1-1)b : \nonumber\\ D \in \C \text{ such that } f_D \text{ is PCF}\}.
        \end{align}
        such that $\sigma^{-1}(\overline{\sigma(W)}) \subseteq V(p)$.

Then we see that $\sigma^{-1}(\overline{\sigma(W)})$ is not Zariski dense in $W = \A^3$. This is a contradiction. Thus, $W = \A^3$ cannot contain a Zariski dense set of PCF maps.\\

Now, let's suppose $\dim(W) = 2$. Suppose that $W$ contains a Zariski dense set of PCF maps and $W$ is not a subvariety of $V(b)$. The same argument as above implies that $\overline{\sigma(W)}$ contains a Zariski dense set of special points. Then, \cite[Theorem 1.2]{DM25} implies that $\overline{\sigma(W)}$ is defined by two polynomials in $\mathcal{C}$.

Note that equivalently there exists a pair of polynomials $g_1, g_2 \in \mathcal{C}$ such that \begin{align}
    \{\sigma^{-1}(g_1), \sigma^{-1}(g_2)\} \subseteq \mathcal{L},
\end{align}
and $\sigma^{-1}(\overline{\sigma(W)}) \subseteq V(\sigma^{-1}(g_1), \sigma^{-1}(g_2))$.
If $$\sigma^{-1}(g_1), \sigma^{-1}(g_2) \in \mathcal{L}\setminus \{ d_1 - d^2_1 - D, (2d_1 - 1)b : D \in \C \text{ such that } f_D \text{ is PCF}\},$$ then $\sigma^{-1}(g_1), \sigma^{-1}(g_2)$ are both irreducible, as $b \neq 0$, and, since they are distinct, they do not share a non-trivial greatest common divisor and hence $$\dim(V(\sigma^{-1}(g_1), \sigma^{-1}(g_2))) \leq 1 .$$

Now, suppose at least one of $\sigma^{-1}(g_1), \sigma^{-1}(g_2)$ is in $$\mathcal{L}_2 \coloneq\{ d_1 - d^2_1 - D, (2d_1 - 1)b : D \in \C \text{ such that } f_D \text{ is PCF}\}.$$ Without loss of generality, suppose it is $\sigma^{-1}(g_1)$. Then if $\sigma^{-1}(g_2) \notin \mathcal{L}_2$, then again, we have $\sigma^{-1}(g_2)$ is irreducible, as $b \neq 0$, and $$\dim(V(\sigma^{-1}(g_1), \sigma^{-1}(g_2))) \leq 1 .$$ 
If both of them live in $\mathcal{L}_2$, then 
$$ V(\sigma^{-1}(g_1), \sigma^{-1}(g_2)) = V(d_1 - d^2_1 - D, (2d_1 - 1)b).$$
Since $b \neq 0$, we have $d_1 = 1/2$. But $D = 1/2 - 1/4 = 1/4$ does not satisfy that $f_D$ is PCF. Hence, in this case $V(\sigma^{-1}(g_1), \sigma^{-1}(g_2)) = \emptyset$ and it contradicts the assumption that $\dim(W) = 2$.\\

Lastly, suppose $\dim(W) = 1$, then $\dim(\overline{\sigma(W)}) = 1$. We denote $\pi_{i}$ denote the projection from $\A^4$ to the $i$-th coordinate, where $i \in \{1,2,3,4\}$ and similarly for any subset $S \subseteq \{1,2,3,4\}$, we denote $\pi_S$ the projection, $\A^4 \to \A^{\# S}$, to the coordinates in $S$. 

Let's first suppose that $$\pi_1 : \overline{\sigma(W)} \to \A^1$$ is dominant. Then the projection $\pi_{1,2}(\overline{\sigma(W)})  $ is a curve that contains a Zariski dense set of special points, which implies, by \cite[Theorem 1.2]{DM25} that either 
\begin{equation}
    d_1 b + c -  D_0 = 0
\end{equation}
with some $D_0 \in \C$ such that $f_{D_0}$ is a PCF polynomial, or 
\begin{equation}
    d_1b + c + d^2_1 - d_1 = 0.
\end{equation}

\textbf{Case (1):} Suppose $d_1b +c - D_0 = 0$ then the $\dim\left(\pi_{1,3}(\overline{\sigma(W)})\right) = 1$ gives that, by \cite[Theorem 1.2]{DM25}, either 
\begin{equation}
    (1-d_1)b +c - D_1 =0,
\end{equation}
with some $D_1 \in \C$ such that $f_{D_1}$ is a PCF polynomial, or 
\begin{equation}
    -d^2_1 +d_1- (1-d_1)b - c = 0.
\end{equation}
\textbf{Subcase (i):} If $(1-d_1)b + c - D_1 = 0$ holds, then plug in $c = D_0 - d_1b$, we have 
\begin{equation}\label{eq: dimW=1-d-1}
    (2d_1 - 1)b = D_1 - D_0.
\end{equation}
Now, since the curve $\pi_{1,4}(\overline{\sigma(W)})$ also contains a Zariski dense set of special points, by \cite[Theorem 1.2]{DM25}, we have 
either $b$ is a non-zero constant or $$-d^2_1 +d_1 - b = 0.$$ Note that both case together with Equation (\ref{eq: dimW=1-d-1}) will imply that $d_1$ can only take finitely many values. This contradicts our assumption.

\textbf{Subcase (ii):} If $-d^2_1 + d_1 - (1-d_1)b - c = 0$, then again plug in $c = D_0 - d_1b$, we have 
\begin{equation}\label{eq: dimW=1-d-2}
    -d^2_1 +d_1 -b + 2d_1b - D_0 = 0.
\end{equation}
Again the fact that $\pi_{1,4}(\overline{\sigma(W)})$ contains a Zariski dense set of special points implies that either $b$ is a non-zero constant or $-d^2_1 +d_1 -b = 0$. In both cases, together with Equation (\ref{eq: dimW=1-d-2}), we have $d_1$ can only take finitely many values. This contradicts our assumption.\\

\textbf{Case (2):} Now, we suppose that $d_1b + c + d^2_1 - d_1 = 0$. Then again, the fact that $\pi_{1,3}(\overline{\sigma(W)})$ contains a Zariski dense set of special points implies that either 
\begin{equation}\label{eq: dimW=1-d-3}
    (1 - d_1)b + c = D_0
\end{equation}
with some $D_0 \in \C$ such that $f_{D_0}$ is a PCF polynomial, or 
\begin{equation}\label{eq: dimW=1-d-4}
    -d^2_1 +d_1 - (1-d_1)b - c = 0.
\end{equation}

\textbf{ Subcase (i):} Suppose Equation (\ref{eq: dimW=1-d-3}) holds for some $D_0 \in \C$ as above. Then, plug this back into $d_1b + c+d^2_1 - d_1 =0$, we get 
\begin{equation}\label{eq: dimW=1-d-5}
    d_1b + D_0 - b(1-d_1) - d^2_1 +d_1 = 0.
\end{equation}
Now the fact that $\pi_{1,4}(\overline{\sigma(W)}))$ is a curve containing a Zariski dense set of special points implies that either $b$ is a non-zero constant or $-d^2_1 +d_1 -b = 0$. Note that if $b$ is a constant then Equation (\ref{eq: dimW=1-d-5}) implies that there are only finitely many values that $d_1$ can take. This contradicts our assumption. 

On the other hand, if $-d^2_1 +d_1 - b = 0$, then plug in Equation (\ref{eq: dimW=1-d-5}), we have 

$$ d^3_1 -d^2_1 + D_0/2 =0,$$
which implies that $d_1$ can only take finitely many values. Again, this is a contradiction.\\

Now, we suppose that $\pi_1(W)$ is a finite set. Then there exists a $D_0$ such that $f_{D_0}$ is a PCF polynomial such that $\pi_1(\overline{\sigma(W)}) \subseteq V(d_1-d^2_1 -D_0)$ since $W$ is irreducible. Also, $$\pi_{2,3,4}(\overline{\sigma(W)}) \subseteq \A^2$$ 
is one-dimensional by assumption. Then $\pi_{2,4}(\overline{\sigma(W)}) \subseteq \A^2$ must also be a one-dimensional subvariety since otherwise it will imply that $b$ and $c$ can only take finitely many values in $\overline{\sigma(W)}$ contradicting the dimension assumption. Hence \cite[Theorem 1.2]{DM25} implies that one of the following holds 
\begin{enumerate}
    \item $b = D_1$;
    \item $d_1 b + c = D_1$;
    \item $(d_1 - 1)b + c = 0 $;
\end{enumerate}
where $D_1 \in \C$ satisfying that $f_{D_1}$ is a PCF polynomial. 

Suppose $(1)$ happens. Then note that $\pi_{2,3}(\overline{\sigma(W)})$ must also be one-dimensional. Again, \cite[Theorem 1.2]{DM25} implies that either one of $d_1b + c$ and $(1-d_1)b +c$ is a constant or $b(2d_1 -1) = 0$. Since $x^2 + 1/4$ is not a PCF polynomial, we have $d_1 \neq 1/2$ and also $b \neq 0$, we have $b(2d_1 - 1) \neq 0$. Thus, the condition implies that $c$ is also a constant, which contradicting that $\pi_{2,3}(\overline{\sigma(W)})$ is one-dimensional. 

Suppose $(2)$ holds. Then note that $\pi_{3,4}(\overline{\sigma(W)})$ is one-dimensional, which implies that either one of $b$ and $(1-d_2)b + c$ is a constant or $c =bd_1$. If $b$ is a constant, then by the assumption that $d_1b + c = D_1$, we have $c$ is also a constant and hence $\dim(W) = 0$, which is a contradiction. If $(1-d_1)b + c = D_2$ for some $D_2 \in \C$, then we have 
$$ b(1-2d_1) = D_2 -D_1.$$
Again, since $d_1 \neq 1/2$, we have $b$ is a constant and hence $c$ is a constant contradicting $\dim(W) = 1$. Lastly, if $c = d_1b$, then we have 
$$ 2d_1b = D_1 .$$
If $d_1 \neq 0$, then $b = D_1/(2d_1)$ which again implies $\dim(W) = 0$. If $d_1 = 0$, then $c = 0$ and we are in the exceptional locus where
$$F(x,y) = (x^2, y^2 + bx).$$

Finally, suppose $(3)$ holds. Note that $\pi_{2,3}(\overline{\sigma(W)})$ is one-dimensional and hence  either one of $(1-d_1)b + c$ and $d_1b + c$ is a constant or $c = d_1b$. If $(1-d_1)b + c = D_2$ for a constant $D_2 \in \C$, then $c = D_2/2$. If $d_1 \neq 1$, then $b$ is again a constant and this contradicts $\dim(W) = 1$. If $d_1 = 1$, then $d = d_1 - d^2_1 = 0$ and $c = 0$. In this case, $W$ in the exceptional locus and we obtain that 
$$ F(x,y) = (x^2, y^2 + bx).$$
Now, suppose $d_1b + c = D_2$ for some constant $D_2 \in \C$. Then, together with $(d_1 - 1)b + c = 0$, we have $b = -D_2$ and also $c = (d_1 -1)D_2$. This contradicts that $\dim(W) = 1$. Lastly, suppose $c =d_1b$. Then, again with (3), we have $(2d_1 -1)b = 0$. Since $d_1 \neq 1/2$, we have $b = 0$ and $W$ is in the exceptional locus where 
$$ F(x,y) = (x^2 + d, y^2 +c)$$
is a split morphism.

\end{proof} 

\subsection{Proof of Theorem \ref{thm: pcf-subvariety}}

The following proposition solves the last case left:

\begin{prop}\label{prop: pcf-aneq0}
 Let $\mathcal{M}$ denote a subspace of the moduli space of conjugacy classes of degree-$2$ 
polynomial skew products, where each class admits a representative of the form
\[
F(x,y) = (x^2 + d,\; y^2 + ax^2 + bx + c),
\qquad d,c \in \C,~ a,b \in \C^*,
\]
so that $\mathcal{M}$ is naturally identified with $(\C^*)^2 \times \A^2$ via the parameters $(a,b,c,d)$.
   Suppose $W \subseteq \mathcal{M}$ is an irreducible subvariety of positive dimension. Then $W$ does not contain a Zariski dense set of PCF points. 
\end{prop}

\begin{proof}
    Note that $F(x,y)$ is post-critically finite only if $V(y)$ is preperiodic under the action of $F(x,y)$. Denote $H_\infty = \P^2 \setminus \A^2$. \\

    {\bf Step \RNum{1}:} We first rule out the case that $V(y) \cap H_\infty$ is a superattracting periodic point.
    For the purpose of contradiction, we suppose $V(y)$ intersects with $H_\infty$ at a point which is a super-attracting periodic point of $F|_{H_\infty}$. Note that 
    $$ F|_{H_\infty}([x:y:0]) = [x^2: y^2 + ax^2 : 0].$$
    We identify the map $F|_{H_\infty}$ with $f_a(z) \coloneq z^2 + a$ by viewing $H_\infty$ as a projective line. Then $V(y) \cap H_\infty$ is identified as the point $z = 0$. Let $k$ be the period of $0$ under $f_a$ and note that $k > 1$ as $a \neq 0$. Then we have 
    $$ f^k_a (0) = 0,$$
    $$ f^i_a(0) \neq 0,$$
   where $1 \leq i < k$.

    Note that we assumed $b \neq 0$. Along the line $V(y)$, the map $F$ satisfies
\[
F(x,0) = \bigl(x^{2}+d,\; ax^{2}+bx+c\bigr).
\]
Note that
\[
f_d^{n}(x)
   = x^{2^{n}} + 2^{n} d\, x^{2^{n}-2} + o\!\left(x^{2^{n}-2}\right),
\]
for all $n \in \N$, where $o(x^\ell)$ denotes a polynomial whose terms all have degree $<\ell$.

A direct induction shows that
\[
F^{k}(x,0)
   =\left(
        f_d^{k}(x),\;
        b\,2^{k-1}\!\left(\prod_{i=1}^{k-1} f_a^{i}(0)\right) x^{2^{k}-1}
        + o\!\left(x^{2^{k}-1}\right)
     \right),
\]
since $f^k_a(0) = 0$.

Applying $F$ once more yields
\[
F^{k+1}(x,0)
   =\left(
        f_d^{k+1}(x),\;
        a\bigl(f_d^{k}(x)\bigr)^{2}
        + b^{2}2^{2k-2}
          \left(\prod_{i=1}^{k-1} f_a^{i}(0)\right)^{2}
          x^{2^{k+1}-2}
        + o\!\left(x^{2^{k+1}-2}\right)\right). \]

We next expand $f_d^{k+2}(x)$:
\[
f_d^{k+2}(x)
   = (f_d^{k+1}(x))^{2} + d
   = \bigl(f_d^{k}(x)\bigr)^4 + O\!\left(x^{2^{k+1}}\right)
   = \bigl(f_d^{k}(x)\bigr)^4 + o\!\left(x^{2^{k+2}-2}\right).
\]
Thus
\[
F^{k+2}(x,0)
   =\left(
        f_d^{k+2}(x),\;
        f_a^{2}(0)\bigl(f_d^{k}(x)\bigr)^4
        + b^{2}2^{2k-1} f_a(0)
           \left(\prod_{i=1}^{k-1} f_a^{i}(0)\right)^{2}
           x^{2^{k+2}-2}
        + o\!\left(x^{2^{k+2}-2}\right)
     \right).
\]
Then, it is not hard to see that 
\[
\begin{aligned}
F^{2k-1}(x,0)
   = \Bigg(
        f_d^{2k-1}(x),\;
        & f_a^{k-1}(0)\bigl(f_d^{k}(x)\bigr)^{2^{k-1}} \\
        &+\, b^{2} 2^{3k-4}
           \left(\prod_{j=1}^{k-2} f_a^{j}(0)\right)
           \left(\prod_{i=1}^{k-1} f_a^{i}(0)\right)^{2}
           x^{2^{2k-1}-2} \\
        &+\, o\!\left(x^{2^{2k-1}-2}\right)
     \Bigg).
\end{aligned}
\]

More generally, for integers $n\ge 1$ and $0\le m<k$,
\[
f_d^{nk+m}(x)
   = \bigl(f_d^{nk}(x)\bigr)^{2^m}
     + O\!\left(x^{2^{nk+m}-2^{nk}}\right)
   = \bigl(f_d^{nk}(x)\bigr)^{2^m}
     + o\!\left(x^{2^{nk+m}-2^{n-1}}\right).
\]

We proceed on by induction on $n$. Assuming that 

\[
F^{nk}(x,0)
   =\left(
        f_d^{nk}(x),\;
        b^{2^n} 2^{g^{n}(0)\,k - g^n(0)}
        \left(\prod_{i=1}^{k-1} f_a^{i}(0)\right)^{g^{n}(0)}
        x^{2^{nk}-2^{\,n-1}}
        + o\!\left(x^{2^{nk}-2^{\,n-1}}\right)
     \right),
\]
where $g(x) = 2x + 1$, a similar computation gives us
\[
\begin{aligned}
F^{nk+1}(x,0)
   = \Bigg(
        f_d^{nk+1}(x),\;
        & f_a(0)\bigl(f_d^{nk}(x)\bigr)^2 \\
        &+\, b^{2^{n+1}} 2^{2g^n(0)k - 2g^n(0)}
           \left(\prod_{i=1}^{k-1} f_a^{i}(0)\right)^{2g^n(0)}
           x^{2^{nk+1}-2^{n}} \\
        &+\, o\!\left(x^{2^{nk+1}-2^{n}}\right)
     \Bigg).
\end{aligned}
\]

and

\[
\begin{aligned}
F^{(n+1)k-1}(x,0)
   = \Bigg(
       & f_d^{(n+1)k-1}(x),\;
         f_a^{k-1}(0)\bigl(f_d^{nk}(x)\bigr)^{2^{k-1}} \\
        &+\, b^{2^{n+1}} 2^{g^{n+1}(0)k- 2g^{n}(0)-2   }
           \left(\prod_{j=1}^{k-2} f_a^{j}(0)\right)
           \left(\prod_{i=1}^{k-1} f_a^{i}(0)\right)^{2g^n(0)}
           x^{2^{(n+1)k-1}-2^{n}} \\
        &+\, o\!\left(x^{2^{(n+1)k-1}-2^{n}}\right)
     \Bigg).
\end{aligned}
\]

From here, applying $F$ one more time, we obtain that 
\[
\begin{aligned}
F^{(n+1)k}(x,0)
   = \Bigg(
        & f_d^{(n+1)k}(x), \\
        & b^{2^{\,n+1}}
          2^{\,g^{n+1}(0)\,k - g^{n+1}(0)}
          \left(\prod_{i=1}^{k-1} f_a^{i}(0)\right)^{g^{n+1}(0)}
          x^{2^{(n+1)k}-2^{\,n}} \\
        &\qquad +\, o\!\left(x^{2^{(n+1)k}-2^{\,n}}\right)
     \Bigg).
\end{aligned}
\]
This completes the induction argument and we conclude that the general formula is

\[
F^{nk}(x,0)
   =\left(
        f_d^{nk}(x),\;
        b^{2^n} 2^{g^{n}(0)\,k - g^{n}(0)}
        \left(\prod_{i=1}^{k-1} f_a^{i}(0)\right)^{g^{n}(0)}
        x^{2^{nk}-2^{\,n-1}}
        + o\!\left(x^{2^{nk}-2^{\,n-1}}\right)
     \right),
\]
for every $n \in \N^+$.

    If $V(y)$ is preperiodic, then let $V(P(x,y))$ be a periodic subvariety under $F^k$ living in the forward orbit of $V(y)$. There are infinitely many $n_0 \in \N$ such that $$F^{n_0k}(V(y)) = V(P(x,y)).$$ Let $m = \deg(P)$ and $l= \deg_x(P_m(x,y)) \geq 0$, where $P_m(x,y)$ is the homogenuous degree $m$ term in $P(x,y)$. 
    Let $n$ be a positive integer large enough such that $F^{nk}(V(y)) = V(P(x,y))$ and $2^{nk} > m2^{n-1}$.
    However, we will then have that 
    \begin{align}
        P(F^{nk}(x,0)) \text{ has non-vanishing leading term of degree at least } \nonumber \\l2^{nk}+ (m-l)(2^{nk}-2^{n-1}) > (m-1)2^{nk},
    \end{align} 
    which is strictly greater than the degree of the rest of terms. This contradicts the fact that 
    $$P(F^{nk}(x,0)) \text{ is constantly zero} $$
    from the assumption that $F^{nk}(V(y)) = V(P(x,y))$. Thus, we ruled out this case, i.e., if $F(x,y)$ is PCF, then $V(y)$ cannot pass through the super-attracting periodic point of $F|_{H_\infty}$.\\

    {\bf Step \RNum{2}:}
    Now, suppose $V(y)$ intersecting $H_\infty$ at the point $[1:0:0]$ which is not periodic, then the periodic point in the forward image of $[1:0:0]$ under $F|_{H_\infty}$ is non-superattracting. Then if $V(y)$ is preperiodic under $F$, then the periodic curves in the forward orbit of $V(y)$ are lines since they intersecting $H_{\infty}$ at a single non-superattracting periodic point under $F|_{H_\infty}$ and, by \cite[Lemma 5.11]{Xie23}, such an intersection is transverse. 

    Since, we know that the curve in the periodic cycle in the forward images of $V(y)$ are all lines, and notice that these cannot be vertical lines, we can assume that there exists a line $C = V(y - kx -l)$ for some $k \in \C^*$ and $l \in \C$ such that the forward images of this line under $F$ are all lines.

    Now, 
    \begin{equation}\label{eq: generic-translation-map}
        F(C) = \overline{\{(x^2 + d, (k^2 + a)x^2 + (2kl + b)x + l^2 + c) : x \in \C\}}
    \end{equation} is a line implies 
    \begin{equation}\label{eq: aneq0-1}
        2kl + b = 0.
    \end{equation} 

    We consider the space $\C^* \times \A^1 \times (\C^*)^2 \times \A^2$ parametrized by $(k,l, a,b,c,d)$ and denote $$\pi_0: \C^* \times \A^1 \times (\C^*)^2 \times \A^2 \to \C^*$$ and $$\pi_1: \C^* \times \A^1 \times (\C^*)^2 \times \A^2 \to \A^1$$ the projection to the first two coordinates respectively and $\pi_2$ the projection to the rest of the coordinates. 

    Let $$\tilde{W} \subseteq \C^* \times \A^1 \times (\C^*)^2 \times \A^2$$ denote the subvariety consisting of the points $t_0 \in \C^* \times \A^1 \times (\C^*)^2 \times \A^2$ such that 
    $$ \deg\left(F^n_{\pi_2(t_0)}\left(V(y - \pi_0(t_0)x - \pi_1(t_0))\right)\right) = 1$$
    for all $n \in \N$, where $F_{\pi_2(t_0)}$ is the quadratic regular polynomial skew products parametrized by the last four coordinates of $t_0$. Note that, by the discussion above, if $W$ contains a Zariski dense set of PCF points, then $W \subseteq \pi_2(\tilde{W})$.
    
    If $\pi_2(\tilde{W})$ is of dimension $0$, then we are done as this implies that there doesn't exist a positive dimension subvariety in $(\C^*)^2 \times \A^2$ that contains a Zariski dense set of PCF points. 

    Now, we assumed that $\dim (\pi_2(\tilde{W})) \geq 1$. Let $F_{\pi_2(t)}$ denote the generic fiber of $\Phi$ over $\pi_2(\tilde{W})$,
    \begin{equation}
        \Phi: \pi_2(\tilde{W}) \times \A^2 \to \A^2, \quad\Phi(\pi_2(t_0),x_0,y_0) = F_{\pi_2(t_0)}(x_0,y_0),
    \end{equation}
    where $t_0 \in \tilde{W}$ and $(x_0,y_0) \in \A^2$. Let $k(t)$, $l(t)$ denote the generic point of $\pi_0(\tilde{W})$ and $\pi_1(\tilde{W})$ respectively. \\
    
    {\bf Case \RNum{1}:} Suppose $k(t)$ is not preperiodic under $F_{\pi_2(t)}|_{H_\infty}$. By the definition of $\tilde{W}$ and Equation (\ref{eq: generic-translation-map}), we have that 
    $$ F_{\pi_2(t)}(V(y - k(t)x - l(t))) = V(y - (k^2(t) + a(t))x - (l(t)^2 + c(t) - (k^2(t) + a(t))d(t)) ),$$
    where $(a(t), b(t), c(t), d(t))$ denotes the generic points in $\pi_2(\tilde{W})$.  Moreover, $$G^n(k(t), l(t)) \in V(2kl + b(t))$$  for all $n \in \N$, where 
    $$ G(k,l) = (k^2 + a(t), l^2 - d(t)k^2 + c(t) -a(t)d(t)),$$
    since $F^n_{\pi_2(t)}$ maps $V(y - k(t)x - l(t))$ to lines for any $n \in \N$ and from above we know that this requires $2k(t)l(t) +b(t) = 0$.
    Since we assumed that $k(t)$ is not preperiodic under $F_{\pi_2(t)}|_{H_\infty} (k)  = k^2 + a(t)$, we have 
    $$ \Orb_{G}(k(t),l(t)) \subseteq V(2kl+ b(t))$$
    and, hence,
    \begin{equation}
        G^n(V(2kl + b(t))) \subseteq V(2kl + b(t)).
    \end{equation}
    Since $V(2kl + b(t)) \cap H_\infty = \{[1 : 0: 0], [0:1:0]\}$ and 
    $$ G([1:0:0]) =[1 : -d(t): 0] ,$$
    we have this implies that $d = 0$. Hence 
    \begin{equation}
        G(k,l) = (k^2 + a(t), l^2 + c(t)).
    \end{equation}\\

    To simplify the notation, we write $a,b$ and $c$ to denote the coordinates $a(t), b(t)$ and $ c(t)$ of the generic point $t$ in $\tilde{W}$ respectively. Note that $b\neq 0$. Then by \cite[Theorem 1.4]{GNY19} if $a, c \notin \{0,-2\}$, then $V(2kl + b)$ is invariant under $G(k,l)$ implies that $a = c$.  Suppose first that $a, c \notin \{0,-2\}$. We have that $V(2kl + b)$ is invariant under $G$ is equivalent to 
    \begin{equation}
        \frac{b^2}{4l^2} + a = \frac{b}{2(l^2 + a)}
    \end{equation}
    for any $l \notin \{ 0\} \cup V(l^2 + a)$. Clearing the denominator and comparing the coefficients of the $l^4$ terms, we have $a = 0$, which contradicts our assumption. 

    On the other hand, if one of $a, c$ is in $  \{0,-2\}$. This is equivalent to say that one of $k^2 + a$ and $l^2 + c$ is a power map or Chebyshev polynomial. Then the proof of \cite[Theorem 1.3]{GNY19} implies that both of them are simultaneously power maps or Chebyshev polynomials. Since we have the assumption that $a \neq 0$, hence $a = c = -2$. Then again, $V(2kl + b)$ is invariant under 
    $$ G(k,l) = (k^2 -2, l^2 -2)$$
    gives the condition that 
    \begin{equation}
        \frac{b^2}{4l^2} -2 = \frac{b}{2(l^2 -2)}
    \end{equation}
    should hold for any $l \in \C \setminus \{0, \pm \sqrt{2}\}$. This implies that there are only finitely many possible values of $b$ that can potentially make $F(x^2, y^2-2x^2 + bx-2)$ post-critically finite. This contradicts the assumption that $\tilde{W}$ is positive dimension.\\

    {\bf Case \RNum{2}:} Now, suppose that $k(t)$ is preperiodic under $F_{\pi_2(t)}$. Then, we may change the variable to let $k(t)$ denote a point in its forward orbit under $F_{\pi_2(t)}$ on the generic fiber so that $k(t)$ is periodic.
   
Then there exists a $n \in \N^+$ such that 
    $$ (k(t), a(t)) \in V(f_{a(t)}^{ n}(k(t)) - k(t) ) .$$

    Note that in this case, we further restrict $\tilde{W}$ to a subvariety of $ \tilde{W} \subseteq \C^* \times \A^1 \times (\C^*)^2 \times \A^2$ such that for a Zariski dense set of $t_0 \in \tilde{W}$, we have $V(y - \pi_0(t_0)x - \pi_1(t_0))$ is periodic under $F_{\pi_2(t_0)}$ of period dividing $n$. We abuse notation to still denote this subvariety as $\tilde{W}$. This will imply that for a generic $t \in \tilde{W}$, we also have $V(y- \pi_0(t)x - \pi_1(t))$ is of period dividing $n$ under $F_{\pi_2(t)}$. Note that since $W$ is assumed to contain a Zariski dense set of PCF points and the periodic cycles in the forward orbits of $V(y) \cap H_\infty$ does not contain $[1:0:0]$, we still have $W \subseteq \pi_2(\tilde{W})$ by \cite[Lemma 5.11]{Xie23}.

Let $\sigma' : \A^4 \to \A^4$ be defined as 
$$ \sigma'(d_1, a ,b,c) = (d_1(1-d_1), a, b,c). $$
Then we still have $\sigma'^{-1}(W) \subseteq \sigma'^{-1}(\pi_2(\tilde{W}))$ and without making any confusion, we will abuse notation to let $W$ denote $\sigma'^{-1}(W)$ and similarly $\tilde{W}$, $\pi_2$ are with respect to the coordinates $(k,l,d_1,a, b, c)$. 

Since, $V(y - kx - l)$ is a periodic curve of period dividing $n$ under $F_{\pi_2(t_0)}$ for a Zariski dense set of $t_0 \in \tilde{W}$, let $d_1$ be a variable such that $d= d_1(1-d_1)$, we have that every $(k,l,a,b,c,d_1) \in \tilde{W}$ satisfies
\begin{equation}\label{eq: a-d_1-b-c-k-l}
    f^n_{ad^2_1 + bd_1 + c}(d_1 k + l) = d_1 k + l,
\end{equation}
\begin{equation}\label{eq: a-(1-d_1)-b-c-k-l}
    f^n_{a(1-d_1)^2 + b(1-d_1) + c}((1-d_1) k + l) = (1 -d_1) k + l.
\end{equation}
Since both of the equations are monic in $c$, together with 
\begin{equation}\label{eq: a-k}
    f^n_a(k) = k
\end{equation}
and
\begin{equation}\label{eq: b-l-k}
    2kl + b = 0,
\end{equation}
we have that $\dim(\tilde{W}) \leq 3$.\\

 Then Equations (\ref{eq: a-k}) and (\ref{eq: b-l-k}) imply that $\dim(\pi_2(\tilde{W})) = 3$ and in particular $(a,b,d_1)$ are free variables. Let $\tau_{S} $, where $S \subseteq \{a,b,c,d_1\}$, denote the projection to the set of coordinates specified by $S$. This implies that $\dim(\tau_{a,b,d_1}(\pi_2(\tilde{W}))) = 3 $. By \cite[Theorem 1.2]{DM25}, we know that one of the following expression must be $0$: 
\[
\mathcal{L} \coloneq 
\left\{
\begin{aligned}
& a - D, \\
& d_1(1 - d_1) - D, \\
& a d_1^2 + b d_1 + c - D, \\
& a(1 - d_1)^2 + b(1 - d_1) + c - D, \\
& d_1(1 - d_1) - a, \\
& d_1(1 - d_1) - (a d_1^2 + b d_1 + c), \\
& d_1(1 - d_1) - \bigl(a(1 - d_1)^2 + b(1 - d_1) + c\bigr), \\
& a(1 - d_1)^2 + b(1 - d_1) + c - a, \\
& a d_1^2 + b d_1 + c - a, \\
& a - 2 a d_1 - 2 b d_1 + b : D \in \C \text{ such that $f_D$ is PCF}
\end{aligned}
\right\}.
\]
If any $p \in \mathcal{L}$ that doesn't involve $c$ holds, then $\dim(\tau_{a,b,d_1}(\pi_2(\tilde{W}))) \leq 2 $, which is a contradiction.

Suppose $a d_1^2 + b d_1 + c = D$, for some $D$ such that $f_D$ is PCF. We obtain that, by Equation (\ref{eq: a-d_1-b-c-k-l}), $f^n_D(d_1k + l) = d_1k + l$ and hence $d_1$ is determined by $a$ and $b$ as $k$ and $l$ are determined by them. This again implies that $\dim(\tau_{a,b,d_1}(\pi_2(\tilde{W}))) \leq 2 $. The exact same argument also works for the case that one of 
$$ \{d_1(1 - d_1) = (a d_1^2 + b d_1 + c),  a d_1^2 + b d_1 + c = a\}$$
holds. Now, similarly, applying the same argument with Equation (\ref{eq: a-(1-d_1)-b-c-k-l}), we  can obtain contradictions when one of the following equations holds
\[
\left\{
\begin{aligned}
& a(1-d_1)^2 + b(1-d_1) + c = D, \\
& d_1(1-d_1) = a(1-d_1)^2 + b(1-d_1) + c, \\
& a(a-d_1)^2 + b(1-d_1) + c = a \; ; \\
& D \text{ is a constant such that } f_D \text{ is PCF}
\end{aligned}
\right\}.
\]
Thus, we conclude that $\dim(\tilde{W}) \leq 2$.\\

\textbf{The case when $\dim(W) \leq 2$:}

Since Equation (\ref{eq: a-d_1-b-c-k-l}) implies that $c$ is determined by $a, b,d_1$, we have 
$\dim( \tau_{a,b,d_1}(\pi_2(\tilde{W}))) \leq 2$. Thus, there are two expressions in $\mathcal{L}$ holds simultaneously by \cite[Theorem 1.2]{DM25}. \\

\textbf{Subcase (1):} If two expressions in 
$$ \mathcal{L}' \coloneq  \{a - D,d_1(1 - d_1) - D,d_1(1 - d_1) - a: D \in \C \text{ such that $f_D$ is PCF}\}$$
hold, then there are only finitely many values that $a$ and $d_1$ can take. Hence, Equations (\ref{eq: a-d_1-b-c-k-l}) and (\ref{eq: a-(1-d_1)-b-c-k-l}) imply that $b$ can only take finitely many values as well. This contradicts that $\tilde{W}$ is of positive dimension.\\

\textbf{Subcase (2):} Now, suppose there exists a pair of $p_1 \in \mathcal{L}'$ and $p_2 \in \mathcal{L} \setminus \mathcal{L}'$ that equal to $0$. If $p_2$ is not $a-2ad_1 - 2bd_1 + b$, then one of $ad_1^2 + bd_1 + c$ and $a(1-d_1)^2 + b(1-d_1) + c$ is completely determined by $a$ or $d_1$ (or just a constant). Since the argument will be exactly the same, we demonstrate one case here. Without loss of generality, assume $p_2$ is $ad^2_1 + bd_1 + c -a$. Then Equation (\ref{eq: a-d_1-b-c-k-l}) becomes
$$ f^n_a(d_1k + l) = d_1k + l.$$
Also, $p_1$ gives that either $a$ (or $d_1$) takes only finitely possible values or $a$ is determined by $d_1$. Now, these together with Equations (\ref{eq: a-k}) and (\ref{eq: b-l-k}) imply that there is only finitely many values that $b$ can take given a $d_1$ (or $a$). Hence, we have $\tilde{W}$ is at most dimension $1$ in this case.

If $p_2 $ is $a-2ad_1 - 2bd_1 + b$, then $p_2$ and $p_1$ together imply that $$\dim(\tau_{a,b,d_1}(\pi_2(\tilde{W})) )\leq 1, $$
since $f_{1/4}$ is not a PCF map and $W $ cannot live in $V(d_1 -1/2)$. Hence $\dim(\pi_2(\tilde{W})) \leq 1$ and $\dim(\tilde{W}) \leq 1$.\\

\textbf{Subcase (3):} Now, suppose both $p_1, p_2 \in \mathcal{L} \setminus \mathcal{L}'$. If 
\[
p_1, p_2 \in \mathcal{L}_1 \coloneq 
\left\{
\begin{aligned}
& a d_1^2 + b d_1 + c - D, \\
& a d_1^2 + b d_1 + c - a, \\
& d_1(1-d_1) - (a d_1^2 + b d_1 + c)
\;:\;
D \in \C \text{ such that } f_D \text{ is PCF}
\end{aligned}
\right\}.
\]
or 
\[
p_1, p_2 \in \mathcal{L}_2 \coloneq
\left\{
\begin{aligned}
& a(1-d_1)^2 + b(1-d_1) + c - D, \\
& d_1(1-d_1) - \bigl(a(1-d_1)^2 + b(1-d_1) + c\bigr), \\
& a(1-d_1)^2 + b(1-d_1) + c - a \\
&\quad :\; D \in \C \text{ such that } f_D \text{ is PCF}
\end{aligned}
\right\}.
\]
then we have that either $a$ or $d_1$ takes a fixed value or $a$ is determined by $d_1$. Moreover, without loss of generality suppose $p_1,p_2 \in \mathcal{L}_1$. Then Equation (\ref{eq: a-d_1-b-c-k-l}) becomes 
$$ f_q^n(d_1k - b/2k) = d_1k - b/2k,$$
where $q$ is a constant or $q \in \{a,d_1(1-d_1)\}$. This gives a non-trivial relation between $a$ and $b$ as the left hand side has greater than $1$ $b$-degree and $k$ is determined by $a$. Hence, 
$\tau_{a,b,d_1}(\pi_2(\tilde{W}))$ is of dimension less than or equal to $1$ and so is $\tilde{W}$. A similar argument also works with Equation (\ref{eq: a-(1-d_1)-b-c-k-l}) if $p_1,p_2 \in \mathcal{L}_2$.

Now, suppose $p_1 \in \mathcal{L}_1$ and $p_2 \in \mathcal{L}_2 \cup \{a - 2 a d_1 - 2 b d_1 + b\}$. The other case that $p_1 \in \mathcal{L}_1 \cup \{a - 2 a d_1 - 2 b d_1 + b\}$ and $p_2 \in \mathcal{L}_2$ can be handled in exactly the same way and so we omit it here. We first suppose that $d_1 \neq 1/2$. Then $p_1$ and $p_2$ together imply that 
$$ b =P(a,d_1)/(1- 2d_1),$$
where $P$ is a polynomial and $\deg_{d_1} (P) \leq 2$ and its $d_1$-degree $2$ term has coefficient  either $0$ or $1$. Thus,
$$ kd_1 +l = kd_1 -b/2k = kd_1 - P(a,d_1)/(2k(1-2d_1)) = h(k)d_1 + o(d_1),$$
where $o(d^m_1)$ denotes the term of $d_1$-degree smaller than $m$ for any $m \in \N^+$ and $h(k)$ is a non-constant rational function in $k$. Then, Equation (\ref{eq: a-d_1-b-c-k-l}) gives that 
$$ f^n_q(kd_1 +l) =f^{n-1}_{q}(h(k)^2d^2_1 + q + o(d^2_1))= q_0(k)d^{2^n}_1 + o(d^{2^n}_1) $$
$$= d_1k + l = h(k)d_1 + o(d_1), $$
where $q_0$ is a non-constant rational function, $q$ is either a constant or $q \in \{a,d_1(1-d_1)\}$.
Thus, this is a non-trivial relation between $d_1$ and $a$ unless $\pi_0(\tilde{W}) \subseteq  V(q_0)$. But, note that if $\pi_0(\tilde{W}) \subseteq   V(q_0)$, then $a$ can also only take finitely many values by Equation (\ref{eq: a-k}). Therefore, $p_1$ and $p_2$ will again imply that $b$ is determined by $d_1$. Thus $\dim(\tau_{a,b,d_1}(\pi_2(\tilde{W}))) \leq 1$ and hence $\tilde{W}$ is of dimension less or equal to $1$.

Now, suppose $d_1 = 1/2$. Then Equation (\ref{eq: a-d_1-b-c-k-l}) implies that 
$$ f^n_q(k/2 - b/2k) = k/2- b/2k,$$
where $q $ is either a constant or $q=a$. In either case, we have that $b$ can only take finitely many values given a value of $a$. Thus $\dim(\tau_{a,b,d_1}(\pi_2(\tilde{W}))) \leq 1$.\\

\textbf{Case when $\dim(W) \leq 1$:}

Now, we have $\dim(\tau_{a,b,d_1}(\pi_2(\tilde{W}))) \leq 1$ by Equation (\ref{eq: a-d_1-b-c-k-l}). Note that $\tilde{W}$ cannot live in $V(d_1 - 1/2)$ as $f_{1/4}$ is not a PCF map. \\

\textbf{Subcase (1):} We first look at the case that $d_1 \in V(d_1) \cup V(d_1 -1)$. Since these two cases can be handled in the same way, we assume, without loss of generality, that $d_1 = 1$.  

Let $\sigma : \A^4 \to \A^4$ be a morphism defined by
$$ \sigma(d_1,a,b,c) = (d_1(1-d_1), a, ad^2_1 + bd_1 + c, a(1-d_1)^2 + b(1-d_1) + c).$$
The same argument as in Propositions \ref{prop: b=0=special-subvariety} and \ref{prop: pcf-a=0} gives that for a point $(d_1,a,b,c) \in W$, if $$F_{d_1(1-d_1),a,b,c} \coloneq (x^2 + d_1(1-d_1), y^2 + ax^2 + bx +c)$$ is PCF then $\sigma(d_1,a,b,c)$ is a special point.
Let $\mu_S : \A^4 \to \A^1$ be the projection from $\A^4$ to the coordinates marked by $S$, where $S \subseteq \{1,2,3,4\}$. If $\tau_a(W)$ is a finite set. Then, simply apply \cite[Theorem 1.2]{DM25} to $\mu_{2,3} \circ \sigma (W)$ and $\mu_{2,4}\circ \sigma(W)$. We will get that $b$ and $c$ are also determined by $a$ and $d_1$ and thus $\dim(W) = 0$. Now, we assume $\tau_a(W)$ is an infinite set.

Then $\dim(\mu_{2,3}\circ \sigma(W)) = 1$ implies, by \cite[Theorem 1.2]{DM25}, that 
either $a + b + c = a$ or $a +b + c = D$ for a constant $D$ such that $f_D$ is PCF.
On the other hand, $\dim(\mu_{2,4} \circ \sigma(W))=1$ implies that either $c$ is a constant or $a = c$. Suppose $c = D$ for some constant $D$. Then Equation (\ref{eq: a-(1-d_1)-b-c-k-l}) implies that $-b/2k$ can only take finitely many values. Now, since $a + b + c$ is either $a$ or a constant, we have $b$ is linear in $a$. Thus, $a$ is linear in $k$. Now, Equation (\ref{eq: a-k}) implies that $k$ can take only finitely many values and so is $a$. This is a contradiction.

On the other hand, if $a + b + c = D$ for a constant and $c=a$. Then, $b = D$ and Equation (\ref{eq: a-d_1-b-c-k-l}) directly give that $k$ can only take finitely many values and again a contradiction.

 For the case that $a = c$ and $a + b + c = a$, we have $a = c = -b$. Now, the periodic cycle of $V(y -kx -l)$ consisting of only lines implies that $F^2(V(y-kx -l))$ is a line, which is equivalent to 
$$ (k^2+a)(l^2 + a) + b= 0.$$
Together with Equation (\ref{eq: b-l-k}), we obtain that 
\begin{equation}\label{eq: F^2-a-k}
    a\left((k^2 + a)(a/(2k)^2 + 1) -1\right) = a\left(\frac{a^2}{4k^2} + \frac{5}{4}a + k^2 -1\right) = 0.
\end{equation}
Similarly, $F^3(V(y- kx -l))$ is a line implies that 
\begin{equation}\label{eq: F^3-a-k}
\left((k^2+a)^2+a\right) \left(((a/2k)^2+a)^2+a\right)-a = 0.
\end{equation}
Since $$\frac{a^2}{4k^2} + \frac{5}{4}a + k^2 -1$$ does not divide the LHS of Equation (\ref{eq: F^3-a-k}) and it is irreducible, we have Equations (\ref{eq: F^2-a-k}) and (\ref{eq: F^3-a-k}) together imply that $a$ can only take finitely many values.\\

\textbf{Subcase (2):} Now, we suppose that $\tilde{W}$ lives in $V(d_1 - D_0)$ for some constant $D_0 \in \C$ such that $f_{D_0}$ is PCF but $D_0 \neq 0,1$. Note that if $a$ is also a constant, similarly as in subcase $(1)$, applying \cite[Theorem 1.2]{DM25} to $\mu_{2,3}(\sigma(W))$ and $\mu_{2,4}(\sigma(W))$ will easily give that $b$ and $c$ are also constant. Hence, $W$ is not positive dimensional. Suppose $a$ is not constant. Then, applying \cite[Theorem 1.2]{DM25} to $\mu_{2,3}(\sigma(W))$ and $\mu_{2,4}(\sigma(W))$ will give us the following three cases up to the symmetry of interchanging the role of $d_1$ and $1-d_1$: 
\begin{enumerate}
    \item $ ad^2_1 + bd_1 +c = D_1, \quad a(1-d_1)^2 + b(1-d_1) + c = D_2;$
    \item $ad^2_1 + bd_1 + c  = a, \quad a(1-d_1)^2 + b(1-d_1) + c = D_2;$
    \item $ ad^2_1 + bd_1+ c = a = a(1-d_1)^2+ b(1-d_1) + c,$
\end{enumerate}
where $D_1$ and $D_2$ are constants such that $f_{D_1}$ and $f_{D_2}$ are PCF maps. 
In the first two cases, Equation (\ref{eq: a-(1-d_1)-b-c-k-l}) implies that 
$$ f^n_{D_2}\left(k(1-d_1) - \frac{b}{2k} \right) = k(1-d_1) - \frac{b}{2k}.$$
Hence,
\begin{equation}\label{eq: b-k}
    k(1-d_1) - \frac{b}{2k}
\end{equation}

can only take finitely many values. Note that we also have in case $(1)$ that 
$$ b = -a + C_1 $$
and in case $(2)$ that 
$$ b = \frac{2 -2d_1}{2d_1 -1}a +C_2$$
where $C_1$ and $C_2$ are two constants. Note that $d_1 \neq 1/2$ implies that the leading coefficients are non-zero. Hence, plugging in these to Expression (\ref{eq: b-k}), we have 
$$a = C_3k^2 + C_4$$
for both cases where $C_3$ is some non-zero constant and $C_4$ is a constant.
Now, Equation (\ref{eq: a-k}) gives that 
$$ f^n_{C_3k^2 + C_4}(k) = k,$$
which is not an equation that constantly hold as it has a non-vanishing degree $1$ term and others are even degree. Thus, there are only finitely many $k$ can make it hold and hence $a,b,c$ can only take finitely many values, contradicting that $W$ is positive dimension. 

Now, we look at case $(3)$. Similarly, we have $b = -a$ and $c = (1 -d^2_1 + d_1)a$ in this case. Since we have the assumption that $F^2(V(y-kx+a/2k))$ is again a line as $W \subseteq \pi_2(\tilde{W})$, we have 
\begin{equation}
    (k^2+a)\left(-d_1(1-d_1)(k^2+a) + \frac{a^2}{4k^2} + (1-d^2_1 + d_1)a\right) =a 
\end{equation}
which expands to 
$$4d_1(1-d_1)k^6 + 4a (d^2_1 - d_1 +1)k^4 + 5a^2k^2 + a(a^2-1) = 0 .$$
The LHS is an irreducible polynomial in $\C[a][k]$ (from the Eisenstein's criterion). Again, it doesn't divide $f^n_a(k) - k$ as the later contains a $k$ term with coefficient $1$. Thus, there are only finitely many values that $a$ and $k$ can take. Therefore, $W$ is not positive dimension.\\

\textbf{Subcase (3):} Now we suppose $a$ is a constant that makes $f_a$ a PCF map. If $d_1$ is also a constant, then we arevback to a case handled above and conclude that $W$ must be dimension $0$. Now, suppose $d_1$ is not a constant. Then, again applying \cite[Theorem 1.2]{DM25}, we have the following three cases up to the symmetry of interchanging the roles of $d_1$ and $1-d_1$: 
\begin{enumerate}
    \item $ ad^2_1 + bd_1 + c = D_1, \quad a(1-d_1)^2 + b(1-d_1) + c = D_2;$
    \item $ ad^2_1 + bd_1 + c = d_1(1-d_1), \quad a(1-d_1)^2 + b(1-d_1) + c = D_2;$
    \item $ ad^2_1 + bd_1 + c =d_1(1-d_1) = a(1-d_1)^2 + b(1-d_1) + c,$
\end{enumerate}
where $D_1$ and $D_2$ are constant such that $f_{D_1}$ and $f_{D_2}$ are PCF maps. 

Suppose $(1)$ holds. Then we can solve that $b = -a + C_1/(2d_1-1)$ for some constant $C_1$. Then Equation (\ref{eq: a-d_1-b-c-k-l}) gives that 
$$ f^n_{D_1}(kd_1 - (-a + C_1/(2d_1-1))/(2k)) =kd_1 - (-a + C_1/(2d_1-1))/(2k), $$
which implies that 
$$ kd_1 - (-a + C_1/(2d_1-1))/(2k)$$
takes only finitely many possible values. Since $k \neq 0$, we have $d_1$ can only take finitely many values. This implies that $W$ is a finite set.

Suppose $(2)$ holds. Then we can solve that 
$$ b = \frac{d^2_1 - d_1 + D_2}{1-2d_1} -a.$$
Plugging this into Equation (\ref{eq: a-(1-d_1)-b-c-k-l}) and clear the denominators, we obtain that 
$$ (4k^2-1)d^2_1 + (1 - 6k^2 - 2a)d_1 + o(d_1) = 0,$$
where $o(d_1)$ is a polynomial in $\C[k]$. Other than the case that $k^2 = 1/4$ and $a = 1/4$, we have that $d_1$ can only take finitely many values. While, since $f_{1/4}$ is not a PCF map, we have that $a \neq 1/4$. Thus, $d_1$ can only take finitely many values, which implies that $W$ is a finite set.

Suppose $(3)$ holds. Then, we can solve that $b = -a$ and $$c = (-1-a)d^2_1 + (a+1)d_1.$$ Since $F^2(V(y-kx-l))$ is again a line, we have that, similarly as the previous computations,
\begin{equation}\label{eq: (3)-F^2-lines}
    4k^2(k^2+a)(k^2-1)d^2_1 - 4k^2(k^2+a)(k^2-1)d_1 + a[a(k^2+a) + 4k^2] = 0.
\end{equation} 
Since $k \neq 0$ and also the periodic cycle of $k$ under $f_a$ doesn't contain $0$, we have that if $k \neq \pm 1$, then this implies that $d_1$ can only take finitely many values. Now, suppose $k = 1$. If $n = 1$, then Equation (\ref{eq: a-d_1-b-c-k-l}) implies that 
$$ a(4d_1 + a -2) = 0.$$
Since $a \neq 0$, we have $d_1$ can only take finitely many values. Now, if $n > 1$, then we have 
$$ f^2_{d_1(1-d_1)}(kd_1 + a/2k)  = 16a(a+2)d^2_1 + o(d_1^2),$$
where $o(d^2_1)$ is a polynomial in $\C[k][d_1]$ of $d_1$-degree smaller than $2$. Note that if $a \neq -2$, then for any $n \geq 2$, we have 
$$\deg_{d_1} (f^n_{d_1(1-d_1)}(kd_1 + a/2k))  = 2^{n-1}.$$
Thus, Equation (\ref{eq: a-d_1-b-c-k-l}) implies that $d_1$ can only take finitely many values. If $a = -2$, then note that the periodic cycle in the forward orbit of $0$ under $f_{-2}$ is $2$ which is a fixed point. Then the set of values that $d_1$ can take making $F$ PCF will in particular satisfy that 
$ F^2(V(y - 2x -1/2))$ is again a line by Equation (\ref{eq: b-l-k}) and \cite[Lemma 5.11]{Xie23}. Again, this implies that, by plugging in values of $a= -2$, $b =2 $ and $k = 2$ into Equation (\ref{eq: (3)-F^2-lines}),
$$ 96d^2_1 + o(d^2_1) = 0,$$
which gives that $d_1$ can only take finitely many values. 

Now, if $k=  -1$, then we compute that 
$$ f_{d_1(1-d_1)}(kd_1 + a/2k)- kd_1 - a/2k = (a+2)(4d_1 + a)$$
and 
$$f^2_{d_1(1-d_1)}(kd_1 + a/2k) =16a(a+2)d^2_1 + o(d^2_1) .$$
Then using exactly the same argument as above, we have that if $a \neq -2$ then $d_1$ can only take finitely many values. If $a = -2$, the above argument again tells us that there are only finitely many $d_1$ that can make $F$ a PCF map. Then $W$ is a finite set. \\

\textbf{Subcase (4):} Now, we are left with the case that $\dim(\mu_{1,2}(\sigma(W)))=1$ and none of $a, d_1$ are constant. In this case, \cite[Theorem 1.2]{DM25} implies that $a = d_1(1-d_1)$. Then we apply \cite[Theorem 1.2]{DM25} to $\mu_{1,3}(\sigma(W))$ and $\mu_{1,4}(\sigma(W))$ and we have that one of the three following cases hold up to the symmetry of interchanging the role of $d_1$ and $1- d_1$: 
\begin{enumerate}
    \item $ ad_1 + bd_1 + c = D_1, \quad a(1-d_1)^2 + b(1-d_1) + c = D_2;$
    \item $ ad^2_1 + bd_1 + c = d_1(1-d_1), \quad a(1-d_1)^2 + b(1-d_1) + c = D_2;$
    \item $ ad^2_1 + bd_1 + c =d_1(1-d_1) = a(1-d_1)^2 + b(1-d_1) + c,$
\end{enumerate}
Suppose $(1)$ or $(2)$ holds. We can solve that $$b = (-2d_1^3 + o(d^3))/(1-2d_1).$$
Now, Equation (\ref{eq: a-k}) becomes 
$$ f^n_{d_1(1-d_1)}(k) = k.$$
When $d_1$ approaches infinite, any branch of the curve $V(f^n_{d_1(1-d_1)}(k) -k)$, denoted as $k(d_1)$, will be a power series of $d_1$ such that $\deg_{d_1}(k(d_1)) = 1$. Now, Equation (\ref{eq: a-(1-d_1)-b-c-k-l}) also implies that 
\begin{equation}\label{eq: d_1-k-D_2}
    f^n_{D_2}(k(1-d_1) - (-2d^3_1 + o(d^3_1))/(2k(1-2d_1)) = k(1-d_1) - (-2d^3_1 + o(d^3_1))/(2k(1-2d_1)).
\end{equation}
Now, $\deg_{d_1}(k(d_1)) = 1$ implies that 
$$ \deg_{d_1}(k(1-d_1) - (-2d^3_1 + o(d^3_1))/(2k(1-2d_1)) = 2$$
as $d_1$ approaches infinity.
Thus, Equation (\ref{eq: d_1-k-D_2}) implies that $d_1$ can only take finitely many values. This implies that $W$ is a finite set.

Suppose $(3)$ holds. Then we can solve that 
$$ b = -a = -d_1(1-d_1).$$
Again, let $k(d_1)$ be an arbitrary branch of the curve 
$$V(f^n_{d_1(1-d_1)}(k) -k)$$
as $d_1$ approaches infinite and we have $\deg_{d_1}(k(d_1)) = 1$ as a power series in $d_1$. Now, Equation (\ref{eq: a-d_1-b-c-k-l}) becomes 
$$ f^n_{d_1(1-d_1)}(kd_1 + d_1(1-d_1)/(2k)) = kd_1 + d_1(1-d_1)/(2k).$$
Plug in $k(d_1)$, we obtain that 
$\deg_{d_1}(k(d_1)d_1 + d_1(1-d_1)/(2k(d_1))) = 2$ and hence Equation (\ref{eq: a-d_1-b-c-k-l}) implies that there are only finitely many values that $d_1$ can take. Hence, we conclude that $W$ is a finite set.

\end{proof}
We repeat here the statement of Theorem \ref{thm: pcf-subvariety} for the reader's convenience.
\begin{thm}[Theorem \ref{thm: pcf-subvariety}]
     Let $\mathcal{M}$ denote the moduli space of conjugacy classes of degree-$2$ polynomial skew products, where each class admits a representative of the form
\[
F(x,y) = (x^2 + d,\; y^2 +ax^2+ bx + c),
\qquad d, a, b,c \in \C,
\]
so that $\mathcal{M}$ is naturally identified with $\A^4$.
Let $W \subseteq \A^4$ be an irreducible Zariski closed subset of dimension at least $1$.
If $W$ contains a Zariski dense set of post-critically finite (PCF) points, then $W$ lives in the exceptional locus
\[
\bigl( V(b) \cap V(a) \bigr)\ \cup\ \bigl(V(a) \cap V(d)\cap V(c)\bigr)\ \cup \bigl(V(b) \cap V(c) \cap V(d) \bigr).
\]
\end{thm}
\begin{proof}
    First, Proposition \ref{prop: pcf-aneq0} concludes that 
    $$W \subseteq V(a) \cup V(b). $$

    If $W \subseteq V(a)$, then Proposition \ref{prop: pcf-a=0} concludes that 
    $$ W \subseteq V(a) \cap \bigl(V(b) \cup \bigl( V(d) \cap V(c) \bigr)\bigr).$$
    Similarly, if $W \subseteq V(b) \setminus V(a)$, then Proposition \ref{prop: b=0=special-subvariety} concludes that 
    $$ W \subseteq V(d) \cap V(b) \cap V(c).$$
    These concludes the proof.
\end{proof}
\begin{lem}\label{lem: a=0-bneq0-subvariety}
   There are infinitely many $b \in \C$ such that $$F_b(x,y) = (x^2, y^2 + bx)$$
   is a post-critically finite endomorphism on $\P^2$.
\end{lem}
\begin{proof}
    It is enough to show that there are infinitely many $b \in \C$ such that $V(x) \cup V(y)$ are preperiodic under $F_b$. Note that it is obvious that $V(x)$ is invariant under $F_b$ for all $b \in \C$. Thus, we only need to show that the set of $b \in \C$ such that $V(y)$ is preperiodic under $F_b$ is infinite. 

    A direct computation shows that 
    $$ F^n_b(V(y)) = \overline{\left\{  \left(x^{2^n}, f_b^{n}(0)x^{2^{n-1}}\right) : x \in \C  \right\}},$$
    where $f_b(z) = z^2 + b$.
    Hence, if $0$ is preperiodic under $f_b$, then $V(y)$ is preperiodic under $F_b$. Since there are infinitely many $b \in \C$ such that $f_b$ is post-critically finite, we have that there exist infinitely many $b \in \C$ such that $0 $ is preperiodic under $f_b$ and so $V(y)$ is preperiodic under $F_b$.

    This concludes the proof.
\end{proof}

\begin{rmk}
Theorem~\ref{thm: pcf-subvariety}, together with
Lemma~\ref{lem: a=0-bneq0-subvariety}, implies that a non-isotrivial family
of quadratic polynomial skew products contains a Zariski dense set of
post-critically finite (PCF) endomorphisms if and only if one of the
following holds:
\begin{enumerate}
\item the family consists of homogeneous polynomial endomorphisms;
\item the family is a subfamily of split morphisms consisting of a Zariski dense set of PCF maps;
\item the family is conjugate to the one-parameter family of polynomial
endomorphisms
\[
F_b(x,y) = (x^2,\; y^2 + bx),
\]
parametrized by $b \in \C$.
\end{enumerate}
\end{rmk}
\subsection{Implication on Conjecture \ref{conj-const-1.2}}
We demonstrate here that our main theorem in this section proves a special case of Conjecture \ref{conj-const-1.2}, which is Corollary \ref{cor: pcf-subv-imply-conj}. Note that the set up of Conjecture \ref{conj-const-1.2} is closely connected to Conjecture \ref{conj: DM24-conj-1,1}. It is actually implied by Conjecture \ref{conj: DM24-conj-1,1} since the assumption that a family of subvarieties contains a Zariski dense set of preperiodic subvarieties implies the family contains a Zariski dense set of preperiodic points. 

\begin{proof}[Proof of Corollary \ref{cor: pcf-subv-imply-conj}]
    Note that after a conjugation, we can assume that for every $s \in S$, $F_s$ is given by 
   $$F([x:y:z]) = [x^2 + dz^2: y^2 + ax^2 + bxz + cz^2 : z^2],$$ 
 for some $a,b,c,d \in \C$. 

 The first case is that $\Phi$ is isotrivial in the sense that for every $s \in S$, $F_s $ is constantly equal to 
 $$ F ([x:y:z]) = [x^2 + dz^2: y^2 + ax^2 + bxz + cz^2 : z^2]$$
 for some fixed $a,b,c,d \in \C$. Then our assumption implies that $F_s$ is post-critically finite for all $s \in S$ and, in particular, $V(x) \cup V(y)$ is preperiodic under $F$. 
 
 Then for any $N \in \N$, we have $$\Phi^{\times N} (s, p_1, \dots, p_{2N}) = (s, F(p_1), \dots, F(p_{2N}))$$
 $$ \mathcal{C}^N = S \times \left( V(x) \times V(y) \right)^N.$$
 Also, $\mathcal{C}^N$ is $\Phi^{\times N}$-special since $(V(x) \times V(y))^N$ is preperiodic under $F^{\times 2N}$. Therefore, the relative special dimension of $\mathcal{C}^N$ over $S$ is $r_{\Phi^{\times N}, \mathcal{C}^N} = 2N$. 
 
 Hence, we need to verify that $\hat{T}^{2N}_{\Phi^{\times N}} \wedge [\mathcal{C}^N] \neq 0$.  
 Let $$\pi : S \times (\P^2)^{2N} \to (\P^2)^{2N}$$ denote the projection map. By the projection formula, we have 
 \begin{align}
     \pi_* (\hat{T}^{2N}_{\Phi^{\times N}} \wedge [\mathcal{C}^N])= \pi_*(\hat{T}^{2N}_{\Phi^{\times N}} \wedge \pi^*[(V(x) \times V(y))^N]) \\= \pi_* \hat{T}^{2N}_{\Phi^{\times N}} \wedge [(V(x) \times V(y))^N] = T^{2N}_{F^{\times 2N}} \wedge [(V(x) \times V(y))^N].  \nonumber
 \end{align} 
Let $m_1, m_2$ be a positive integers such that $$(F^{\times 2N})^{m_1}((V(x) \times V(y))^N)$$
is a periodic subvareitey under $F^{\times 2N}$ of period $m_2$.
Let $T = \left(\sum_{i=1}^{2N} \pi_i^*\omega\right)^{\wedge 2N}$ be a $2N$-current on $(\P^2)^{2N}$, where $\pi_i$ denotes projection onto the $i$-th factor and $\omega$ is the Fubini–Study form on $\P^2$. 
    Then
    \[
    \int_{(\P^2)^{2N}} T \wedge [\left(F^{\times 2N}\right)^{m_1}\left((V(x) \times V(y))^N\right)] > 0.
    \]
    Let $d = \deg(F^{\times 2N})$. 
    Observe that for every $n \in \N^+$
    \begin{align*}
       & \int_{(\P^2)^{2N}} d^{-2N(nm_2 + m_1)} \left((F^{\times 2N})^{nm_2 + m_1}\right)^*T \wedge [(V(x) \times V(y))^N] 
        \\&= \int_{(\P^2)^{2N}} d^{-2N(nm_2 + m_1)} T \wedge (F^{\times 2N})^{nm_2 + m_1}_*[(V(x) \times V(y))^N] \\
        &= \int_{(\P^2)^{2N}} d^{-2N(nm_2 + m_1)} T \wedge d^{2N(nm_2 + m_1)}[\left(F^{\times 2N}\right)^{m_1}\left((V(x) \times V(y))^N\right)] \\
        &= \int_{(\P^2)^{2N}} T \wedge [\left(F^{\times 2N}\right)^{m_1}\left((V(x) \times V(y))^N\right)].
    \end{align*}
    Moreover, by the local uniform convergence of the potentials of $d^{-2Nn}(F^{\times 2N})^{n*}T$ to those of $T_{F^{\times 2N}}^{\wedge 2N}$ (see \cite[Chapter~III, Corollary~3.6]{Dem}), we have
    \[
    d^{-2Nn}(F^{\times 2N})^{n*}T \wedge [(V(x) \times V(y))^N] \longrightarrow T_{F^{\times 2N}}^{\wedge 2N} \wedge [(V(x) \times V(y))^N]
    \quad \text{as } n \to \infty.
    \]
    Hence,
    \[
    \int_{(\P^2)^{2N}} T_{F^{\times N}}^{\wedge 2N} \wedge [(V(x) \times V(y))^N] > 0,
    \]
    and therefore,
    \[
    T_{F^{\times 2N}}^{\wedge 2N} \wedge [(V(x) \times V(y))^N] > 0.
    \]

  Therefore, $$\hat{T}^{r_{\Phi^{\times N}, \mathcal{C}^N}}_{\Phi^{\times N}} \wedge [\mathcal{C}^N] \neq 0.$$ \\
 Now, let's suppose $\Phi$ is not isotrivial. Then there exists a morphism $\tau : S \to \A^4_\C$ such that \begin{align}
     \Phi(s,[x_1:y_1:z_1], [x_2:y_2:z_2])\\ = (s, [x^2_1 + \tau_1(s)z^2_1: y^2_1 + \tau_2(s)x^2_1 + \tau_3(s)x_1z_1 + \tau_4(s)z^2_1 : z^2_1], \nonumber\\ [x^2_2 + \tau_1(s)z^2_2: y^2_2 + \tau_2(s)x^2_2 + \tau_3(s)x_2z_2 + \tau_4(s)z^2_2 : z^2_2]).\nonumber
 \end{align}
 Then, our assumption that there exists a Zariski dense set of $s \in S$ such that $F_s$ is post-critically finite implies that $\tau(S)$ contains a Zariski dense set of PCF points. Then Theorem \ref{thm: pcf-subvariety} implies that 
 $$ \tau(S) \subseteq \bigl( V(b) \cap V(a) \bigr)\ \cup\ \bigl(V(b) \cap V(d)\cap V(c)\bigr)\ \cup \bigl(V(a) \cap V(c) \cap V(d) \bigr).$$\\

 {\bf Case \RNum{1}:} Suppose $\tau(S) \subseteq V(a) \cap V(b)$. This implies that \begin{align}
     \Phi_s([x_1:y_1:z_1], [x_2:y_2:z_2])\\ = ( [x^2_1 + \tau_1(s)z^2_1: y^2_1 + \tau_4(s)z^2_1 : z^2_1], \nonumber\\ [x^2_2 + \tau_1(s)z^2_2: y^2_2 +  \tau_4(s)z^2_2 : z^2_2])\nonumber
 \end{align} is a split morphism for every $s \in S$.
Denote  
\[
\Delta \coloneq
\overline{
\left\{
\begin{aligned}
&(x_1,y_1,x_2,y_2,\dots,x_{2N},y_{2N}) \in \mathbb{A}^{4N} : \\
&x_1 = x_3 = \cdots = x_{2N-1}, \\
&y_2 = y_4 = \cdots = y_{2N}
\end{aligned}
\right\}
}
\subseteq (\mathbb{P}^2)^{2N},
\]

and $$\Gamma \coloneq \overline{\bigcup_{s_0 \in S} \Orb_{\Phi_{s_0}}(V(x) \times V(y))}.$$
Then $S \times (\Delta \cap (\Gamma \times (\P^2)^{2N-2}))$ is invariant under $\Phi^{\times N}$ and we claim that $$\dim_S(S \times (\Delta \cap (\Gamma \times (\P^2)^{2N-2}))) \leq 2N + \dim(\tau(S)),$$ and then we have $r_{\Phi^{\times N}, \mathcal{C}^N} \leq 2N+\dim(\tau(S))$.  Note that if $\dim(\tau(S)) = 2$, then this is trivially true as $\dim(\Gamma) \leq 4$ and $$\dim(\Delta \cap (\Gamma \times (\P^2)^{2N-2})) \leq 2N + 2.$$  If $\dim(\tau(S)) = 1$, then since our assumption implies that there exists a Zariski dense set of special points in $\tau(S)$, \cite[Theorem 1.2]{DM25} implies that one of the following holds: 
\begin{enumerate}
    \item $\tau(S) = V(d - D)$;
    \item $\tau(S) = V(c - D)$;
    \item $\tau(S) = V(d -c)$;
\end{enumerate}
where $D \in \C $ such that $f_D$ is post-critically finite. In all these cases, it is obvious that $\dim(\Gamma) = 3$. Hence, $\dim_S(S \times (\Delta \cap (\Gamma \times (\P^2)^{2N-2}))) \leq 2N + 1$ and we verified the claim. \\

 Suppose $\dim(\tau(S)) = 2$.
 For any $N \in \N^+$, let $U = S \times \A^{4N}$, where $\Phi^{\times N}|_U$ is a family of split polynomial endomorphisms. Let $$\pi : S \times (\A^2 \times \A^2)^{N} \to S \times (\A^1 \times \A^1)^N$$ denote the projection map that, on each $\A^2 \times \A^2$ factor, it maps $V(x) \times V(y)$ to $(0,0) \in \A^1 \times \A^1$; i.e., it contracts the second and third coordinates of $\A^2 \times \A^2$. 

 Now, it is sufficient to verify that $\pi_* (\hat{T}^{2N + 2}_{\Phi^{\times N}} \wedge [\mathcal{C}^N] |_{U}) \neq 0$. Note that 
 $$ \mathcal{C}^N|_{U} = \pi^*( S \times \{(0,0)\}^N).$$
Also, $$\hat{T}_1 \coloneq \bigwedge^{N-1}_{i = 0}(p_{4i + 2}^* \hat{T}_{f_{\tau_2(s)}} \wedge p^*_{4i + 3}\hat{T}_{f_{\tau_3(s)}} )\leq \hat{T}^{2N}_{\Phi^{\times N}}|_U,$$
where $p_j: S \times \A^{4N} \to S \times \A^1 $ is the projection to the $j$-th factor. 
Hence, it is sufficient to verify that 
$$ \pi_*(\hat{T}^2_{\Phi^{\times N}} \wedge \hat{T}_1 \wedge \pi^*[S \times \{(0,0)\}^N])|_U  \neq 0.$$

Let $$G(s,x_1, y_1, \dots, x_N, y_N) \coloneq \lim_{n \to \infty} 2^{-n}\log\max_{1 \leq i \leq N}\{1, |f^n_{\tau_1(s)}(x_i)|,|f^n_{\tau_4(s)}(y_i)|\}.$$

Note that, by the projection formula, 
\begin{align}\label{eq: ineq-current-1}
    \pi_*\bigl( \hat{T}_{\Phi^{\times N}}^{\wedge 2} \wedge \hat{T}_1 \bigr)\big|_U
    &= \pi_*\Biggl( \Bigl( \pi^* \hat{T}_{\Psi} 
        + \sum_{i=1}^N \bigl( p_{4i+2}^* \hat{T}_{f_{\tau_2(s)}} 
        + p_{4i+3}^* \hat{T}_{f_{\tau_3(s)}} \bigr) \Bigr)^{\!\wedge 2} 
        \wedge \hat{T}_1 \Biggr)\Big|_U \\[4pt]
    &\geq \pi_*\bigl( (\pi^* \hat{T}_{\Psi})^{\wedge 2} \wedge \hat{T}_1 \bigr)\big|_U 
      \nonumber  \\[4pt]
    &\geq \hat{T}_{\Psi}^{\wedge 2}\big|_U 
       \nonumber
\end{align}
where 
$$\Psi : S \times (\A^2)^N \to S \times (\A^2)^N$$
such that 
$$ \Psi(s, x_1, y_1, \dots, x_N, y_N) = (s, f_{\tau_1(s)}(x_1), f_{\tau_4(s)}(y_1), \dots, f_{\tau_1(s)}(x_N), f_{\tau_4(s)}(y_N))$$
and
$$ \hat{T}_{\Psi}|_U  = dd^c G(s,x_1, y_1, \dots, x_N, y_N). $$
Let $\mu : S \times \A^{2N} \to \A^2 \times \A^{2N}$ be the map given by 
$$ \mu(s,x_1, y_1, \dots, x_N, y_N) = (\tau(s), x_1, y_1, \dots, x_N, y_N).$$

Hence, with the inequality (\ref{eq: ineq-current-1}), it is sufficient to show that 
$$\mu_* \Bigl( \pi_* \bigl( \hat{T}_{\Phi^{\times N}}^{\wedge 2} \wedge \hat{T}_1 \wedge \pi^*[S \times \{(0,0)\}^N] \bigr) \Bigr)\big|_U
\;\geq\;
\mu_* \bigl( \hat{T}_{\Psi}^{\wedge 2} \wedge [S \times \{(0,0)\}^N] \bigr)\big|_U
\;\neq\; 0$$

Let $(s_1, s_2) $ denote the two coordinates of $\tau(S) = \A^2$. Then, we have 
\begin{align}
    \mu_* \left( \hat{T}^2_{\Psi}|_U \wedge [S \times \{(0,0)\}^N] \right)  \geq dd^cG(s_1,s_2,0,\dots,0) \wedge dd^c G(s_1,s_2, 0,\dots,0)
\end{align}
which is a measure on $\A^2$, where
$$ G(s_1,s_2,0, \dots, 0) \coloneq \lim_{n \to \infty} 2^{-n}\log\max \{1, |f^n_{s_1}(0)|,|f^n_{s_2}(0)|\}$$
$$ = \max\{u(s_1),v(s_2)\}$$
and
$$ u(s_1) \coloneq \lim_{n \to \infty} 2^{-n} \log \max\{1, |f^n_{s_1}(0)|\} ,$$
$$v(s_2) \coloneq \lim_{n \to \infty} 2^{-n} \log \max\{1, |f^n_{s_2}(0)|\} .$$

Then $$(dd^c G)^{\wedge2}|_{u(s_1) = v(s_2)} \geq dd^c u \wedge dd^c v |_{u(s_1) = v(s_2)},$$
where the RHS is equal to 
$$ (\mu_{bif} \wedge \mu_{bif}) |_{u(s_1) = v(s_2)},$$
and $\mu_{bif}$ is the bifurcation measure on the Moduli space of quadratic polynomials parametrized by $c \in \C$ of $f_c(x) = x^2 + c$.

Since $$\mathcal{M} \times \mathcal{M} \subseteq \{(s_1,s_2): u(s_1) = v(s_2)\},$$
where $\mathcal{M}$ is the Mandelbrot set of quadratic polynomials, as $u = v = 0$ when restricted on $\mathcal{M}$, and $\mu_{bif}(\mathcal{M}) = 1$, we have 
$$ \mu_{bif} \wedge \mu_{bif}(\mathcal{M} \times \mathcal{M}) = 1.$$
Thus, 
$$(\mu_{bif} \wedge \mu_{bif}) |_{u(s_1) = v(s_2)} \neq 0,$$
and, therefore, 
$$ (dd^c G)^2 \neq 0.$$\\

 Similarly, suppose $\dim(\tau(S)) =1$. Then, it is sufficient to show 
$$\pi_* (\hat{T}^{2N + 1}_{\Phi^{\times N}} \wedge [\mathcal{C}^N] |_{U}) \neq 0 .$$ The exact same argument shows that it is sufficient to verify that
$$ \pi_* \hat{T}_{\Phi^{\times N}}|_U \wedge [S \times \{(0,0)\}^N] \neq 0.$$
Note that 
$$\pi_* \hat{T}_{\Phi^{\times N}}|_U \wedge [S \times \{(0,0)\}^N] = dd^c G(s,0,\dots, 0) \neq 0 ,$$
since $s \to G(s,0,\dots,0)$ is again subharmonic, non-constant and bounded from below.\\

{\bf Case \RNum{2}:} Now, suppose $\tau(S) \subseteq \left( V(b) \cap V(c) \cap V(d)\right) $. 
Let \[
\Delta_y \coloneq
\overline{
\left\{
\begin{aligned}
&(x_1,y_1,\dots,x_{2N},y_{2N}) \in \mathbb{A}^{4N} : \\
&y_{2i}x_{2j} = x_{2i}y_{2j}, \\
&x_{2i-1} = 0, \\
&\forall\, i,j \in \mathbb{N}^+, \; 1 \le i,j \le N
\end{aligned}
\right\}
}
\subseteq (\mathbb{P}^2)^{2N}.
\]

In this case, for every $N \in \N^+$, the subvariety $S \times \Delta_y$ is invariant under $\Phi$. Hence $r_{\Phi^{\times N}, \mathcal{C}^N} \leq 2N + 1$. Thus, it is sufficient to show that 
$$ \hat{T}^{2N + 1}_{\Phi^{\times N}} \wedge [\mathcal{C}^N] \neq 0.$$
Take $U \coloneq S \times (\A^2 \times \P^2 \setminus \{[0:0:1]\})^{N}$. It is sufficient to show that 
$$ \hat{T}^{2N + 1}_{\Phi^{\times N}} \wedge [\mathcal{C}^N]|_{U} \neq 0.$$
Let $$\pi_{H_\infty} : U \to S \times (V(y) \times H_\infty)^{N},$$
and note that, similar to the argument in the above case, 
$$ (\pi_{H_\infty})_* \hat{T}^{2N + 1}_{\Phi^{\times N}} \geq \hat{T}_{\Psi},$$
where $\Psi : S \times (\A^2)^N \to S \times (\A^2)^N$ is given by 
$$ \Psi(s,x_1,y_1, \dots, x_N, y_N) = (s, f_0(x_1), f_{\tau_2(s)}(y_1), \dots, f_0(x_N), f_{\tau_2(s)}(y_N)).$$
Then, by the projection formula,
\begin{align*}
  (\pi_{H_\infty})_* \left (  \hat{T}^{2N + 1}_{\Phi^{\times N}} \wedge [\mathcal{C}^N]|_{U} \right) =  (\pi_{H_\infty})_* \left (  \hat{T}^{2N + 1}_{\Phi^{\times N}} \wedge \pi^*_{H_\infty}[S \times \{(0,0)\}^N]|_{U} \right)\\\geq \hat{T}_{\Psi}\wedge [S \times \{(0,0)\}^{2N}] .  
\end{align*}

Note that 
$$\hat{T}_{\Psi} \wedge [S \times \{(0,0)\}^{2N}] = dd^c G(s,0,\dots 0), $$
where $$G(s,x_1,x_2, \dots, x_{2N}) \coloneq \lim_{n \to \infty } 2^{-n} \log \max_{1 \leq i \leq N} \{1, |f^n_{0}(x_{2i-1})|, |f^n_{\tau_2(s)}(x_{2i})|\}$$ and $x_i$ is in the $i$-th factor of $\A^{2N}$. 

Since $f_{\tau_2(s)}$ is not isotrivial, we have 
 $$s \to G(s,0, \dots, 0)$$ is subharmonic, nonconstant, and bounded from below and hence 
 $$ dd^c G(s,0,0,\dots)\neq 0.$$\\

 {\bf Case \RNum{3}:} Lastly, suppose $\tau(S) \subseteq V(d) \cap V(a) \cap V(c)$. Let 
\[
\Delta_b \coloneq 
\overline{
\left\{
\begin{aligned}
&(x_1,y_1,\dots,x_{2N},y_{2N}) \in \mathbb{A}^{4N} : \\
&y_{2i}x_{2j}^2 = x_{2i}^2 y_{2j}, \\
&x_{2i-1} = 0, \\
&\forall\, i,j \in \mathbb{N}^+, \; 1 \le i,j \le N
\end{aligned}
\right\}
}
\subseteq (\mathbb{P}^2)^{2N}.
\]

 Then $\Phi (S \times \Delta_b) \subseteq S \times  \Delta_b$ and note that each $F_b(x,y) \coloneq (x^2, y^2 + bx) $ maps $V(y^2 - kx)$ to $V(y^2 - f_b(k)x)$ for each $k \in \C$. Hence $$r_{\Phi^{\times N}, \mathcal{C}^N} \leq \dim_S(S \times \Delta_b) = 2N + 1.$$ Thus, it is sufficient to show that 
 $$ \hat{T}^{2N+1}_{\Phi^{\times N}} \wedge [\mathcal{C}^N] \neq 0. $$
 
 Note also that $\Phi^{\times N}_*([\mathcal{C}^N]) = c_1 [\mathcal{D}^N]$,
 where $c_1$ is a non-zero constant and $\mathcal{D}^N \coloneq \overline{\{{s} \times (V(x_1) \times V(y^2_2 - \tau_3(s)x_2))^{N} : s \in S\}} \subseteq S \times (\P^{2})^{2N}$.
 Let $S' \coloneq S \setminus \tau^{-1}_3(0)$ and $U \coloneq S' \times (\A^2 \times (\C^*)^2)^N \subseteq S \times (\P^2)^{2N}$. It is sufficient to show that 
 $$\Phi^{\times N}_* \left( \hat{T}^{2N+1}_{\Phi^{\times N}} \wedge [\mathcal{C}^N]\right) |_{U} = c_2 \left( \hat{T}^{2N+1}_{\Phi^{\times N}} \wedge [\mathcal{D}^N]\right) |_{U} \neq 0,$$
 for some non-zero constant $c_2$.

 Let $\pi : U \to S' \times \A^{2N} $ be the projection map such that on each $\A^2 \times (\C^*)^2$ factor in $U$, we have 
 $$ \pi|_{S' \times \A^2 \times (\C^*)^2 } (s, x_1,y_1, x_2,y_2) = (s, x_1, y^2_2 /x_2). $$
 Then, it is sufficient to show that 
 $$ \pi_*\left(\left(\hat{T}^{2N+1}_{\Phi^{\times N}} \wedge [\mathcal{D}^N]\right) |_{U}\right) \neq 0.$$
 Again, a similar computation as above cases shows that 
 $$ \pi_*\hat{T}_{\Phi^{\times N}}|_U = \sum^{2N}_{i = 1}(p_{2i-1}^*\hat{T}_{f_{\tau_1(s)}}  + p^*_{2i}\hat{T}_{f_{\tau_3(s)}} ),$$
 where $p_i :S' \times \A^{2N} \to S' \times \A^1$ denote the projection to the $i$-th $\A^1$ factor for each $i \in \{1,2, \dots, 2N\}$.
 Thus, by the projection formula, it is sufficient to show that 
$$ \left(\pi_*\hat{T}^{2N+1}_{\Phi^{\times N}}|_U\right) \wedge [\{(s,0,\tau_3(s),0,\tau_3(s), \dots ,0,\tau_3(s): s \in S'\}] \neq 0 .$$
Again, similarly, $\pi_* \hat{T}^{2N+1}_{\Phi^{\times N}} \geq \hat{T}_\Psi$, where $$\Psi : S' \times (\A^2)^N \to S'  \times (\A^2)^N$$ is given by 
$$ \Psi(s, x_1, y_1, \dots, x_N ,y_N) = (s , f_0(x_1), f_{\tau_3(s)}(y_1), \dots, f_0(x_N), f_{\tau_3(s)}(y_N)).$$
Thus, it is sufficient to show that 
$$ \hat{T}_\Psi \wedge [\{(s,0,\tau_3(s),0,\tau_3(s), \dots ,0,\tau_3(s): s \in S'\}] \neq 0$$
Let $f_{i,s} = f_{\tau_1(s)}$ if $i$ is odd, and $f_{i,s} = f_{\tau_3(s)}$ if $i$ is even for $i \in \{1,2, \dots, 2N\}$.
 Since $f_{\tau_3(s)}$ is not isotrivial, we have 
$$ s \to G(s, 0,\tau_3(s), \dots, 0, \tau_3(s)) $$
is non-constant, subharmonic and bounded from below, where 
$$ G(s, x_1, \dots, x_{2N}) \coloneq \max_{1 \leq i \leq 2N}\{G_{i,s}(x_i)\},$$
$$G_{i,s}(x) \coloneq \lim_{n \to \infty} 2^{-n} \log \max\{1 , |f^n_{i,s}(x)|\}.$$

Hence 
\begin{align}
    \hat{T}_\Psi \wedge [\{(s,0,\tau_3(s),0,\tau_3(s), \dots ,0,\tau_3(s): s \in S'\}] \nonumber \\= dd^cG(s, 0, \tau_3(s), \dots, 0, \tau_3(s)) \neq 0.
\end{align}

\end{proof}

\section{Height and  preperiodic points for  families of polynomial skew products}\label{sec:unlikelypolyskew}
\subsection{Absolute values on number fields}\label{subsec:absolutevalues}
	Let $K$ be a number field and $\overline{K}$ be a fixed algebraic closure of $K$. Let $M_K$ be the set of places of $K$, that is, an equivalence class of nontrivial absolute values on $K$. For each place $v\in M_K$, $K_v$ denotes the corresponding completion of $K$ with respect to the absolute value $|\cdot|_v$ (determined up to equivalence). The absolute value $|\cdot|_v$ on $K$ is either standard archimedean absolute value or the $p$-adic absolute value satisfying $|p|_p=p^{-1}$ when restricted to $\Q$. For any non-zero element $x\in K$, the product formula 
	$$\prod_{v\in M_K}|x|^{N_v}_v=1$$ holds.  Here $N_v:=[K_v:\Q_v]$. Denote by $\C_v$ the completion of the algebraic closure $\overline{K}$ of the number field $K$ with respect to the absolute value $|\cdot|_v$. By abuse notation, we still denote by $|\cdot|_v$ the unique extension to $\C_v$ of the absolute value on $K$. Note that $\C_v$ is both complete and algebraically closed. For $(x_1,x_2,...,x_n)\in \C_v^n$, we write $$\|(x_1,x_2,...,x_n)\|_v:=\max\{|x_1|_v,|x_2|_v,...,|x_n|_v\}.$$
	\subsection{Activity and bifurcation} \label{sec:setupactivitybifurcation}
		Let $M=\mathbb{A}^1_{\mathbb{C}}$. 
		We denote by $F_t:=F(t,\cdot):M\times\mathbb{A}^2\rightarrow\mathbb{A}^2$ a family of  regular polynomial skew product of degree $d\geq 2$ parametrized by $t$ given by $$F_t(x,y)=(f_t(x),g_t(x,y))$$ where $f_t(x)=x^d+O(x^{d-1})\in \mathbb{C}[x,t]$ and $g_t(x,y)=y^d+O(y^{d-1})\in\mathbb{C}[x,y,t]$. Let $P_t:M\rightarrow \mathbb{A}^2(\mathbb{C}[t])$ be a marked point defined by $$P_t(t):=(a(t),b(t)).$$ 
		We impose a technical condition on a marked point $P_t$ to ensure that the bifurcation measure associated to $(F_t,P_t)$ is non-trivial and it satisfies:
		\begin{assump}\label{assump:initialmarkedpoint} Suppose that $P_t=(a(t),b(t))$ is a marked point so that the sequence $\{f_t^n(a(t))\}_n$ is not normal for $t\in M$.
		\end{assump}
	 In terminology \cite[section 2.1]{BD13}, we call $a(t)$ active. We observe that the assumption \ref{assump:initialmarkedpoint} also implies that the sequence $\{F^n_t(P_t)\}_n$ is not normal for $t\in M$.  Following \cite{BD13},  the assumption \ref{assump:initialmarkedpoint} implies that the bifurcation measure 
	 $$\mu_{a(t)}:=dd^c_t G(a(t))$$ is nonzero. Here $$G(a(t))=\lim_{n\rightarrow\infty}d^{-n}\log\max\{1,|f^n_t(a(t))|\}.$$ Analogously, we obtain that the measure $$\mu_{F_t}:=dd^c_t G(P_t)$$ is nonzero where $$G(P_t):=\lim_{n\rightarrow\infty}d^{-n}\log\max\{1,\|F^n_t(P_t)\|\}.$$  This  is due to the fact that the support of $\mu_{F_t}$ is the set of $t\in M$ so that $\{t\mapsto F^n_t(P_t)\}_n$ fails to be normal (cf. \cite[Proposition-Definition 3.1, Theorem 3.2]{DF08} and \cite[Lemma 1.3]{Ga23}).
     \begin{rmk} 
     The failure of the sequence  $\{F^n_t(P_t)\}_{n\geq0}$ to be normal does not necessarily imply that sequence of holomorphic family $\{f_t^n(a(t))\}_{n\geq0}$  is not normal. For instance, let us consider a holomorphic family of polynomial skew product $F_t:\mathbb{A}^2_{\mathbb{C}}\rightarrow\mathbb{A}^2_{\mathbb{C}}$ parametrized by $t\in \mathbb{C}$ and defined by $$F_t(x,y)=(x^2,y^2+x+t).$$ Now, given an initial marked point $P_t(t)=(0,0).$ Let $\pi_2$ be the projection map onto the second factor of $\mathbb{A}^2$. It is known that the sequence $$\{\pi_2(F^n_t(P_t))\}_n=\{t,t^2+t,...\}$$ fails to be normal in neighborhood of $t$ intersecting the boundary of the  Mandelbrot set. In contrast, the sequence $$\{f_t^n(a(t))\}_n=\{0,0,...\}$$ forms a normal family. 
     \end{rmk}
  
\subsection{Height function associated to adelic line bundle}
Let $K$ be a number field and let $X$ be a quasi projective variety over $K$.  For each place $v\in M_K$, we denote by $X^{\text{an}}_v$ the Berkovich analytic space at the place $v$. A metrized line bundle on $X^{\text{an}}_v$ is a pair $\overline{L}=(L,\|\cdot\|_v)_{v\in M_K}$ consisting of a line bundle $L$ on $X$ and a continuous metric $\|\cdot\|_v$ on  $X^{\text{an}}_v$. Let $\overline{L}$ be a nef adelic line bundle on $X$. Define a height function $$h_{\overline{L}}: X(\overline{K})\rightarrow\mathbb{R}$$ by $$h_{\overline{L}}(x)=-\frac{1}{[K(x):K]}\sum_{v\in M_K}\sum_{y\in \text{Gal}(\overline{K}/K)\cdot x \times_K K_v}\log\|s(y)\|_v^{\deg_{K_v}(y)}$$ for any nonzero rational section $s$ of $L$ on $X$ with $x\not\in|\text{div}(s)|$ (i.e., $s$ does not vanish at $x$) and $\text{Gal}(\overline{K}/K)\cdot x \times_K K_v$ is the Galois orbit of $x$ in $K_v$.

\begin{rmk} \label{rem:exampleofsetprep}
     It is worth pointing out that the set $\text{Prep}(P_t)$ could be either finite or infinite. Let $K$ be a number field and let us consider the family of polynomial skew product $F_t: \mathbb{P}^2(K[t])\rightarrow\mathbb{P}^2(K[t])$ $$F_t(x,y)=(x^2,y^2+txy).$$ 
Given any nonzero integer $0\neq a\in\mathbb{Z}$, it is straightforward to check that $$\#\text{Prep}((a,0))=\varnothing$$ because the $F_t$-forward orbit of $(a,0)$ is $\{(a^{2^n},0)\}_{n\geq 1}$ which is infinite for all $t\in \overline{K}$. On the other hand, one can check that  $$\#\text{Prep}((0,1))=\infty$$ because $(0,1)$ is fixed by $F_t$ for all $t\in \overline{K}.$
\end{rmk}
In section \ref{sec:prependo}, we provide a necessary condition for which the set $\text{Prep}(P_t)$ is infinite for a given family of polynomial skew product. In fact, we prove a statement for regular endomorphism on $\mathbb{P}^2$. It will be crucial for stating an unlikely intersection result for a family of polynomial skew product in section \ref{sec:unlikelyinter}.
\subsection{Preperiodic points of regular endomorphisms on $\P^2$ in families}\label{sec:prependo}
\begin{prop}\label{sec:necessarycondforinfiniteprep}
    Let $F_t : \P^2_{\overline{\C(t)}} \to \P^2_{\overline{\C(t)}} $ be a surjective endormophism defined over $\overline{\C(t)}$. Let $p_t \in \P^2_{\overline{\C(t)}}$. Suppose $\overline{\Orb_{F_t}(p_t)}$ is a proper Zariski closed subset of $\P^2_{\overline{\C(t)}}$. Denote $Z_t =\overline{\Orb_{F_t}(p_t)} $. Suppose $(F_t|_{Z_t}, p_t, Z_t)$ is non-isotrivial. Then the set 
    $$ \{s \in \C : p_{s} \in \Prep(F_s)\}$$
    is an infinite set, where $p_s$ and $F_s$ denote the point and endomorphism on $\P^1_\C$ obtained by plugging in $t  =s$ to $F_t$ and $p_t$.
\end{prop}

\begin{proof}
     Since we can work with the irreducible components of $Z_t$ separately, we can assume that $Z_t$ is irreducible without loss of generality. If the genus of $Z_t$ as a curve in $\P^2_{\overline{\C(t)}}$ is greater than $1$, then $F_t|_{Z_t}$ is of finite order and so for every $s \in \C$, we would have $p_s \in \Prep(F_s)$. 

Suppose the genus of $Z_t$ is $1$, then $F_t$ is birationally conjugated by a map $\pi_t$ defined over $\overline{\C(t)}$ to a one-parameter family of endomorphisms on an elliptic curve. Suppose there are infinitely many $m \in \N$ such that $F^m_t(p_t) \in I(\pi_t)$. Then $p_t$ is preperiodic under $F_t$ as $I(\pi_t) \cap Z_t$ is a finite set. Suppose otherwise that there exists a $m \in \N$ such that $\Orb_{F_t}(F^m_t(p_t)) \cap I(\pi_t) = \emptyset$, we have $\pi_t(F^m_t(p_t))$ is a point in the Elliptic curve. Then, since $(F_t|_{Z_t}, F^m_t(p_t), Z_t)$ is also non-isotrivial, \cite[Theorem 1.5]{De16} implies that there are infinitely many $s \in \C$ such that $\pi_s(F^m_s(p_s))$ is preperiodic under $\pi_s \circ F_s \circ \pi^{-1}_s$, where $\pi_s$, $F_s$ and $p_s$ denote the morphisms and point defined over $\C$ obtained by plugging $t = s$. 

Now, suppose $Z_t$ is of genus $0$, then there exists a birational map $\pi_t$ defined over $\overline{\C(t)}$ such that $\pi_t(Z_t) = \P^1$. Similarly, if there are infinitely many $m \in \N$ such that $F^m_t(p_t) \in I(\pi_t)$, we have $p_t$ is preperiodic under $F_t$. Suppose on the other side that there exists a $m \in \N$ such that $\Orb_{F_t}(F^m_t(p_t)) \cap I(\pi_t) = \emptyset$. The non-isotrivial assumption implies that $(F_t|_{Z_t}, F^m_t(p_t))$ is conjugated by $\pi_t$ to a non-isotrivial dynamical pairs on $\P^1$. Then \cite[Theorem 1.6]{De16} implies that there are infinitely many $s$ such that $p_s$ is preperiodic under $F_s$.
\end{proof}
\begin{example}
    Consider $F_t(x,y) = (x^2 +2tx + t^2 + 1 -t, y^2)$ and $p_t = (1 + t, 0)$. Then we have $Z_t =\overline{\Orb(F_t(p_t))} = V(y)$. Also, $F_t|_{V(y)}(x) = (x+ t)^2 + 1 - t$ and $p_t$ restricted to $V(y)$ is just $1 + t$. Conjugating by $\pi_t (x) = 1 -t$, we have the dynamical pair $(F_t|_{Z_t}, p_t)$ is conjugated to $(x^2 + 1, 1)$ and so it is isotrivial. We see that for any $s \in \C$, we would have $p_s$ is not a preperiodic point of $F_s$ as $1 $ is not preperiodic under $x^2 + 1$.
\end{example}
\begin{example}
    Consider $F_t(x,y) =(x^2 + t, y^2) $ and $p_t = (t^2, 0)$. Since $(x^2 + t, t^2)$ is non-isotrivial, we have that there are infinitely many $s \in \C$ such that $p_s \in \Prep(F_s)$ by Proposition \ref{sec:necessarycondforinfiniteprep} 
\end{example}

\begin{example}
    Consider $F_t(x,y) = (x^2, y^2 + tx + \pi)$ and $p_t = (t,0)$. Notice that for a $s \in \C$, $p_s \in \Prep(F_s)$ only if $s$ is a root of unity or $s = 0$. However, we have the norm of the second coordinate of $F^n_s(p_s)$ for every $s$ that is a root of unity or zero is strictly increasing as $n $ increases. Therefore, none of $s \in \C$ will make $p_s \in \Prep(F_s)$.

    One can also argue that $\Orb_{F_t}(p_t)$ is not contained in a proper Zariski closed subset of $\P^2_{\overline{\C(t)}}$. Suppose for the purpose of contradiction that it is contained in a Zariski closed subset. Since $t$ is not a preperiodic point for $x^2$, we have that there is a polynomial $P_t(x,y)$ over $\overline{\C(t)}$ with $\deg_y(P_t(x,y)) > 0$ such that for any $n \in \N^+$, we have 
    $$ P_t(F^n_t(p_t)) = 0.$$
    Let's write 
    $P_t(x,y) = \sum_{ i =0}^d a_i(t,x)y^i$, where $a_d(t,x)$ is not constantly zero. Then there exists an infinite subset $S \subseteq \N^+$ such that $a_d(t,x_n(t)) \neq 0$ for any $n \in S$, where we denote $F^n_t(p_t) \coloneq (x_n(t), y_n(t))$.

    Now, for any $n \in S$, we have 
    $$ \deg_\pi(P_t(F^n_t(p_t))) = d2^{n -1},$$
    here $\deg_\pi$ denotes the maximum order of $\pi$.
    Hence, $P_t(F^n_t(p_t)) \neq 0$ for any $n \in S$, which is a contradiction.

\end{example}
\subsection{An unlikely intersection result for a one-parameter families of polynomial skew product}\label{sec:unlikelyinter} 
In this subsection, we first provide a brief discussion of the arithmetic equidistribution for quasi projective varieties developed by Yuan-Zhang \cite{YZ26}. It is a pivotal tool to assert an unlikely intersection for our one-parameter family of polynomial skew product.
\begin{thm} \label{thm:YuanZhangequidistribution} \cite[Theorem 5.4.3]{YZ26} Let $X$ be a quasi projective variety over a number field $K$. Let $\overline{L}$ be a nef adelic line bundle on $X$ (whose generic fiber denoted by $\tilde{L}$) such that $\deg_{\tilde{L}}(X)>0$. Let $\{x_n\}_n$ be a generic sequence in $X(\overline{K})$ such that $h_{\overline{L}}(x_n)$ converges to $h_{\overline{L}}(X)$. Then the Galois orbit of $x_n$ is equidistributed to $$d\mu_{\overline{L},v}=\frac{1}{\deg_{\tilde{L}}(X)}c_1(\overline{L})^{\dim(X)}_v$$ in $X^{\text{an}}_v$ for each place $v\in M_K$.
\end{thm}
Return to our setting. Given a family of regular polynomial skew product $F_t$ parametrized by $t\in \mathbb{A}^1$ of degree $d>1$ and initial point $P_t=(a(t),b(t))$ satisfying assumption (\ref{assump:initialmarkedpoint}). Let us consider $F:\mathbb{A}^1\times \mathbb{P}^2\rightarrow \mathbb{A}^1\times \mathbb{P}^2$ given by $$F(t,[x:y])=(t,[f_t(x):g_t(x,y)]).$$ Define $\overline{L}_F$ as $$\overline{L}_F:=\lim_{i\rightarrow\infty}\frac{1}{d^i}(F^i)^*\overline{\mathcal{O}_{\mathbb{P}^2}(1)}.$$ It follows that $\overline{L}_F$ is well-defined and a nef adelic line bundle by \cite[Theorem 6.1.1]{YZ26}. Note also that $\overline{L}_F$ is the $F$-invariant extension of $\mathcal{O}_{\mathbb{P}^2}(1)$ in the sense of \cite[Theorem 6.1.1]{YZ26}. Let $i:\mathbb{A}^1\rightarrow\mathbb{A}^1\times\mathbb{P}^2$ be a section defined by $$i(t)=(t,[a(t):b(t):1]).$$ Notice that $\overline{L}_{\text{par}}:=i^*\overline{L}_F$ is again nef (as it is a pullback of nef line bundle under morphism). In order to apply Theorem \ref{thm:YuanZhangequidistribution}, the equidistribution of small points to $\overline{L}_{\text{par}}$, it is sufficient to check the following item:\\
(i) (non-degeneracy condition) This can be checked from computing $$\deg_{\widetilde{L}_{\text{par}}}(\mathbb{A}^1)=\widetilde{L}_{\text{par}}=\int_{\mathbb{C}} c_1(\overline{L}_{\text{par}})>0$$ where the positivity follows from our assumption that the sequence of holomorphic endomorphism $\{t\mapsto F^n_t(P_t)\}_n$ is not normal for $t\in \mathbb{C}$ and hence the bifurcation measure is nonzero.\\
(ii) (small height) Suppose that $\{t_n\}_n\subset \overline{K}$ is an infinite distinct sequence of algebraic parameters such that $P_t$ is $F_t$-preperiodic.  It is true that $h_{\overline{L}_{\text{par}}}(\mathbb{A}^1)=0.$ This can be achieved by using the fundamental inequality. From our assumption and the definition of height, we see that$$\lim_{n\rightarrow\infty}h_{\overline{L}_{\text{par}}}(t_n)=0.$$ Invoking the number field case of \cite[Theorem 5.3.3]{YZ26}, we conclude $$h_{\overline{L}_{\text{par}}}(\mathbb{A}^1)\leq \sup_{U\subseteq \mathbb{A}^1}\inf_{t\in U(\overline{\mathbb{Q}})}h_{\overline{L}_{\text{par}}}(t)=0$$  where $U$ runs over all open Zariski open subsets of $\mathbb{A}^1$. Note that  we have employed the fact that $\overline{L}_{\text{par}}$ is nef and degenerate 
as well as the generic sequence $\{t_n\}_n$ of small points.  The nefness of $\overline{L}_{\text{par}}$ also yields $$h_{\overline{L}_{\text{par}}}(\mathbb{A}^1)\geq0$$ by \cite[Theorem 4.1.1]{YZ26}. Therefore, $$h_{\overline{L}_{\text{par}}}(\mathbb{A}^1)=0$$ as desired.\\

\noindent Denote by $$h_{P_t}(t):=h_{\overline{L}_{\text{par}}}(t)$$ our good height on the parameter space $\mathbb{A}^1$.	It follows from  the specialization result of \cite[Lemma 6.2.1]{YZ26} that  $$h_{P_t}(t)=0\Longleftrightarrow F^n_t(P_t)=F^m_t(P_t),\,\,\,0\leq m<n.$$ 
		Thus, we write $$\text{Prep}(P_t):=\{t\in \overline{K}: F_t^m(P_t)=F_t^n(P_t)\,\,\,0\leq m<n\}=\{t\in \overline{K}:h_{P_t}(t)=0\}.$$
		We are ready to prove our main result:
\begin{thm}\label{thm:unlikelypolyskewtwoinitialpoints} Let $K$ be a number field. Suppose that $P_t, Q_t\in \A^2(K[t])$ are satisfied by the assumption \ref{assump:initialmarkedpoint} and the sets $\text{Prep}(P_t)$ and $\text{Prep}(Q_t)$ are infinite. The following are equivalent:
			\begin{enumerate}
				\item[(a)] $\#\text{Prep}(P_t)\cap \text{Prep}(Q_t)=\infty$;
				\item[(b)] $\text{Prep}(P_t)= \text{Prep}(Q_t)$;
				\item[(c)] $h_{P_t}=h_{Q_t}$.
			\end{enumerate}
		\end{thm}
\begin{proof} 		(a)$\Longrightarrow$(c) Suppose that $\{t_n\}_n$ is an infinite distinct sequence in $\overline{K}$ so that $P_t$ and $Q_t$ are both preperiodic points of $F_t$. Then we know that $$\lim_{n\rightarrow\infty} h_{P_{t_n}}(t_n)=0\quad\text{and}\quad\lim_{n\rightarrow\infty} h_{Q_{t_n}}(t_n)=0.$$
		From our discussion above, we have the sequence $\{t_n\}$ is equidistributed with respect to both
		$$d\mu_{\overline{L}_{\text{par},P_t},v}=\frac{c_1(\overline{L}_{\text{par},P_t})_v}{\deg_{\widetilde{L}_{\text{par},P_t}}(\mathbb{A}^1)}\quad\text{and}\quad d\mu_{\overline{L}_{\text{par},Q_t},v}=\frac{c_1(\overline{L}_{\text{par},Q_t})_v}{\deg_{\widetilde{L}_{\text{par},Q_t}}(\mathbb{A}^1)}$$ for each place $v\in M_K$, by applying  Theorem \ref{thm:YuanZhangequidistribution}. Thus it yields the equality of measures \begin{equation}
		    c_1(\overline{L}_{\text{par},P_t})_v=c_1(\overline{L}_{\text{par},Q_t})_v\label{eq:equalmeasures}
		\end{equation} for all places $v\in M_K$. 
        In order to apply the Calabi theorem (cf. \cite[Corollary 2.2]{YZ17} and \cite[Corollary A.6.2]{YZ26}), we fix an embedding of $\mathbb{A}^N$ into $\mathbb{P}^N$ ($N\geq1$) and we work with a perturbed line bundle $$L_{F,\epsilon}:=L_F+\epsilon \mathcal{O}_{\mathbb{P}^2}(1)$$ where $\epsilon>0$. Notice that $L_{F,\epsilon}$ is ample while $L_F$ is  nef. Furthermore, $\overline{L}_{\text{par},\epsilon}=i^*\overline{L}_{F,\epsilon}$ is again ample as it is a pullback of an ample bundle under finite morphism $i$. Using equality (\ref{eq:equalmeasures}), the  measures of the perturbed metrized line bundle $\overline{L}_{\text{par},\epsilon}$ become
        $$c_1(\overline{L}_{\text{par},P_t})_v+\epsilon  \rho_v=c_1(\overline{L}_{\text{par},Q_t})_v+\epsilon  \rho_v$$ where $$\rho_v:=\begin{cases} \omega_{\text{FS}}\, \, \text{the Fubini-study form on}\,\, \mathbb{P}^1,&\mbox{if}\,\,v\,\,\text{is Archimedean} \\ \delta_{\text{Gauss}}\,\, \text{the Dirac measure at Gauss point},&\mbox{if}\,\,v\,\,\text{is non-Archimedean}  \end{cases}$$ 
        see \cite[\S 1.2.2 and Proposition 2.1.1]{CL11}. That is, $$c_1(\overline{L}_{\text{par},\epsilon,P_t})_v=c_1(\overline{L}_{\text{par},\epsilon,Q_t})_v$$
Hence the metrics $\|\cdot\|_{v,P_t,\epsilon}$ and $\|\cdot\|_{v,Q_t,\epsilon}$ are proportional by applying \cite[Corollary 2.2]{YZ17} and \cite[Corollary A.6.2]{YZ26}.  Then there exist constants $c_{v,\epsilon}>0$ so that $$\|\cdot\|_{v,P_t,\epsilon}=c_{v,\epsilon}\|\cdot\|_{v,Q_t,\epsilon}.$$ It is nothing that $$\|\cdot\|_{v,P_t}h_v^{\epsilon}=c_{v,\epsilon}\|\cdot\|_{v,Q_t}h_v^{\epsilon}$$ where $$h_v([t_0:t_1]):=\begin{cases}\frac{|a_0t_0+a_1t_1|_v}{\sqrt{|t_0|_v^2+|t_1|_v^2}}\, \, \text{ Fubini-study metric },&\mbox{if}\,\,v\,\,\text{is Archimedean} \\ \frac{|a_0t_0+a_1t_1|_v}{\max\{|t_0|_v,|t_1|_v\}}\,\, \text{ model metric},&\mbox{if}\,\,v\,\,\text{is non-Archimedean} \end{cases}$$ and $a_0, a_1$ are scalars. Since $h_v^{\epsilon}>0$, we can deduce that $c_{v,\epsilon}$ must be  constants independent of $\epsilon$ and thus $$\|\cdot\|_{v,P_t}=c'_v\|\cdot\|_{v,Q_t}$$ for each $v\in M_K.$ Recall that $h_{\overline{L}_{\text{par},P_t}}$ and  $h_{\overline{L}_{\text{par},Q_t}}$ are calculated by evaluating $-\log\| s\|_{v,P_t}$ and $-\log\|s\|_{v,Q_t}$ at $t$ for any nonzero rational section $s$ of $\mathcal{O}_{\mathbb{P}^1}(1)$ that does not vanish at $t$.  Hence $h_{\overline{L}_{\text{par},P_t}}$ and  $h_{\overline{L}_{\text{par},Q_t}}$ differ by a constant $c_v$ with $c_v=0$ for all but finitely many $v\in M_K$. Since there is a sequence in which both heights converge to the same value and so the constant $c=\sum_{v\in M_K}c_v$ must be $0$. Therefore $h_{\overline{L}_{\text{par},P_t}}=h_{\overline{L}_{\text{par},Q_t}}$ and it is clear that $$h_{P_t}=h_{Q_t}.$$
		(b)$\Longrightarrow$(a) is clear. (c)$\Longrightarrow$(b) follows from a property of $h_{P_t}$ which detects parameters $t$ so that $P_t$ is $F_t$-preperiodic. The proof is completed.
\end{proof}
	We end this subsection with a non-trivial yet straightforward computation to demonstrate an instance of unlikely intersection phenomenon predicted by Theorem \ref{thm:unlikelypolyskewtwoinitialpoints}.
	\begin{example}\label{ex:exampleofunlikelyinter}  Consider a family of polynomial skew product $F_t$ defined over a number field $K$ parametrized by $t\in\overline{K}$ given by 
		$$F_t(x,y)=(x^2,y^2+txy).$$
	Given two initial points $P_t=(t,0)$ and $Q_t=(t+1,0)$ in $K[t]\times K[t]$. Notice that $P_t$ and $Q_t$ satisfy the hypotheses of Theorem \ref{thm:unlikelypolyskewtwoinitialpoints} as \begin{align*}\text{Prep}(P_t)&=\{t: t=0\,\,\text{or}\,\,t\in \mu_{\infty}\}\\  \text{Prep}(Q_t)&=\{t: t=-1\,\,\text{or}\,\,(t+1)\in \mu_{\infty}\}\end{align*}where $\mu_{\infty}$ is the set of of all roots of unity.  Notice that $-2\in \text{Prep}(Q_t)$ but $-2\not\in \text{Prep}(P_t)$.  Hence, the sets $\text{Prep}(P_t)$ and $\text{Prep}(Q_t)$ are both infinite and distinct. However, we can easily check that the intersection $$\text{Prep}(P_t)\cap\text{Prep}(Q_t)$$ is finite and non-empty.   In fact, we require both $t$ and $t+1$ to be in the set $\{0\}\cup \mu_{\infty}$. By a routine case-by-case calculation, we have the following\\ \textbf{Case (i)} : $t=0$ which yields $t+1=1$ and both are preperiodic under $x^2$.\\ \textbf{Case (ii)} : $t+1=0$ and this implies $t=-1$. It is obvious that both $0$ and $-1$ are preperiodic under $x^2$.\\ \textbf{Case (iii)} : $t$ and $t+1$ are roots of unity. They must lie on the unit circle. It is elementary to find complex numbers $t$ such that $|t|=1$ and $|t+1|=1$. We solve algebraically to obtain  $t=e^{i2\pi/3}, e^{i4\pi/3}.$ Thus, we conclude $$\text{Prep}(P_t)\cap\text{Prep}(Q_t)=\left\{-1,0,e^{\frac{i2\pi}{3}},e^{\frac{i4\pi}{3}}\right\}.$$ \\
\indent 	Furthermore, let $P'_t=(-t,0)$ be another initial point which differs from $P_t$ when $t\neq0$. Thus, we clearly see that $$\text{Prep}(P_t)=\text{Prep}(P'_t)=\mu_{\infty}$$ and their orbits collide $$F^n_t(P_t)=F^n_t(P'_t)$$ for all $n\geq1$. This exhibits an orbit relation of $P_t$ and $P'_t$.
	\end{example}

\begin{rmk}	Motivated by the above example, it would be of great interest to describe explicit dynamical relation between $P_t$ and $Q_t$ in light of \cite[Theorem 1.3]{BD13}. Keeping the same notation as in Example \ref{ex:exampleofunlikelyinter} for $F_t$, now we set $P_t=(t+a,0)$ and $Q_t=(t+b,0)$ where $a,b\in\mathbb{Z}$. It would  be intriguing to obtain an effective bound or a uniform bound for the set $\text{Prep}(P_t)\cap \text{Prep}(Q_t)$ provided $\text{Prep}(P_t)\neq \text{Prep}(Q_t)$. This question should be valid for other families of polynomial skew products $F_t$ and different families of initial points  $P_t$ and $ Q_t$.

\end{rmk}

\section{A case of Conjecture \ref{conj: DM24-conj-1,1}}\label{sec:proper-sub}
In this section, our aim is to prove that, under some degree conditions, if there are infinitely many $t_0\in \overline{K}$ such that $P_{t_0}$ is preperiodic under $F_{t_0}$, then the Zariski closure of forward orbit of $P_t\in \mathbb{A}^2({K[t]})$ under the action of $F_t$ is contained in a proper subvariety of $\P^2_{\overline{K(t)}}$. As a corollary, it will also conditionally answer a case of Conjecture \ref{conj: DM24-conj-1,1} when the dynamical systems are given by families of regular polynomial skew products. 

\begin{thm}[Theorem \ref{thm:mainpolyskewinsubvar}]
   Let $K$ be a number field and let 
\[
F_t(x,y)=(f_t(x),g_t(x,y))
\]
be a one-parameter family of regular polynomial skew products of degree $d\ge2$, 
with $f_t(x)$ and $g_t(x,y)$ monic polynomials in $K[t][x]$ and $K[t][x,y]$, 
respectively. 

Given a point $P_t := (a(t),b(t)) \in K[t]\times K[t]$ such that
\begin{enumerate}
\item 
$
\deg(a(t)) > \deg_t(f_t);
$

\item $\deg(b(t))$ is positive and
\[
\deg(b(t)) \ge 
\left(
\frac{\operatorname{lcm}(\deg(b(t)),\,\deg(a(t)))}{\deg(b(t))} + 1
\right)
\bigl(\deg(a(t)) + \deg_t(g_t)\bigr).
\]
\end{enumerate}
   Suppose that there are infinitely many $t_0\in \overline{K}$ for which $P_{t_0}$ is preperiodic under the action of $F_{t_0}$. Then the Zariski-closure forward orbit $\overline{\text{Orb}_{F_t}(P_t)}$ is contained in a proper subvariety of $\mathbb{P}^2_{\overline{K}(t)}$. 
\end{thm}

\subsection{Technical lemmas}
Before we start the proof of Theorem \ref{thm:mainpolyskewinsubvar}, we collect here several key technical lemmas.

We first introduce the vertical B\"ottcher coordinate of a regular polynomial skew products. Let's view $K(t)$ as embedded in $\C$ with $t$ having large enough norm and transcendental
 over $K$. Then, by \cite[Proposition 2.6]{Jo99}, there exists a $R \in \R^*$ such that 
    in the region
    $$ D_2 \coloneq\{(x,y) \in \C^2 : G_{F_t}(x,y) - G_{f_t}(x,y) > R\}$$
    one has the vertical B\"ottcher coordinate $\Phi_x$ satisfying
   \begin{enumerate}
       \item $\Phi_x (y) = y + o(1)$ as $|y| \to \infty$;
       \item $\log |\Phi_x| = G_{F_t}(x,y) - G_{f_t}(x,y)$;
       \item 
       $\Phi_{f_t(x)}(g_t(x,y)) = \Phi_x(y)^d$.
   \end{enumerate}

\begin{lem}\label{lem: linear-bound-c-deg}
   Let $f_t(x)$ be a monic polynomial of degree $d>1$ in $\C[t][x]$, and let
\[
g_t(x,y) = y^d +  \sum_{i=0}^{d-1} \mu_{d-i,t}(x)\, y^{i} \in K[t][x,y].
\]
For each $x$, let $\Phi_x$ denote the vertical B\"ottcher coordinate
associated to $(f_t,g_t)$, defined by
\[
\Phi_x(y) = y + \sum_{i=1}^{\infty} c_{i,t}(x)\, y^{-i},
\]
and satisfying the functional equation
\[
\Phi_{f_t(x)}\!\bigl(g_t(x,y)\bigr) = \bigl(\Phi_x(y)\bigr)^d
\]
for all $(x,y)\in D_2$.

Then each coefficient $c_{i,t}(x)$ lies in $\overline{K}[t][x]$. Moreover, for any
$a(t), b(t)\in K[t]$ with $(a(t), b(t))\in D_2$ and $\deg_t(a(t)) > \deg_t(f_t)$, one has the degree bound
\[
\deg_t\!\bigl(c_{i,t}(a(t))\bigr)
   \le k_3 (i+1) + i k_1,
\]
where $k_3 := \deg_t(a(t))$ and $k_1 := \deg_t(g_t)$.

\end{lem}
\begin{proof}
    We will prove this by induction and we verify the base case when $i =1$ first. Notice that the coefficient of $y^{d-2}$ term in $(\Phi_{a(t)}(y))^d$ is given by
    $$ d c_{1,t}(a(t)).$$ While the coefficient of $y^{d-2}$ term in $$ \Phi_{f_t(a(t))}(g_t(a(t),y))$$
    is 
    $$ \mu_{2,t}(a(t))$$
    whose degree of $t$ is bounded by 
    $$ 2k_3 + k_1,$$
    since $\deg_x(\mu_{i,t}) \leq i$ for each $i \in \{0,1,\dots, d\}$.
    This verifies the base case and, in particular, shows that $c_{1,t}(a(t)) \in \overline{K}[t][x]$.

    Now, for any positive integer $n > 1$, we assume the formula holds for any $i < n$. We want to show that the same upper bound also hold for $i = n$. We look at the coefficients of $y^{d-1-n}$ terms on both sides of the equation. 

    The coefficient of the $y^{d-1-n}$ term in $$\Phi_{f_t(a(t))}(g_t(a(t),y)) = \sum_{i=0}^{d} \mu_{d-j,t}(a(t))y^j + \sum^{\infty}_{i = 1} c_{i,t}(f_t(a(t)))\left( y^d+ \sum_{j=0}^{d-1} \mu_{d-j,t}(a(t)) y^j\right)^{-i}$$
    $$ = \sum_{i=0}^{d} \mu_{d-j,t}(a(t))y^j + \sum^{\infty}_{i = 1} c_{i,t}(f_t(a(t)))y^{-id}\left( 1 + \sum_{j=0}^{d-1} \mu_{d-j,t}(a(t)) y^{j-d}\right)^{-i}$$
    $$ = \sum_{i=0}^{d} \mu_{d-j,t}(a(t))y^j + \sum^{\infty}_{i = 1} c_{i,t}(f_t(a(t)))y^{-id}\left(\sum^{\infty}_{s = 0}\left( -\sum_{j=0}^{d-1} \mu_{d-j,t}(a(t)) y^{j-d}\right)^{s}\right)^i$$
    can be written as
    $$ P(c_{1,t}, \dots, c_{n-1,t}, \mu_{1,t}, \dots, \mu_{d,t}) ,$$
    here $P$ is a polynomial in $(c_{1,t}, \dots, c_{n-1,t}, \mu_{1,t}, \dots, \mu_{d,t})$ which is a summation of terms
    $$ c_{j,t}(f_t(a(t)))C_{k_1, \dots ,k_r}\prod^r_{i = 1} \mu_{k_i,t}(a(t)),  $$
    where $j$ is a positive integer such that $dj \leq n+1 - d$ and $k_i$'s and $r$ are positive integers such that 
    $$
    \sum^r_{i = 1}k_i = n+1-d(j+1)
    ,$$ and $C_{k_1, \dotsm k_r}$ is a constant. Note that 
    $$ \deg_t\left(c_{j,t}(f_t(a(t)))C_{k_1, \dots ,k_r}\prod^r_{i = 1} \mu_{k_i,t}(a(t))\right) $$
    $$\leq \left(d(j + 1) + \sum^r_{i = 1} k_i\right)k_3 + (r + j)k_1  \leq (n+1)k_3 + nk_1 .$$
    Therefore, $\deg_t(P) \leq (n+1)k_3 + n k_1.$

    Now, we look at the coefficient of $y^{d-n-1}$ term  in $(\Phi_{a(t)}(y))^d$, which is given by 
    $$ dc_{n,t}(a(t)) + Q(c_{1,t}, \dots, c_{n-1,t}), $$
    where $Q$ is a polynomial in $c_{1,t}, \dots, c_{n-1,t}$. In particular, $Q$ is a summation of terms 
    $$ C_{l_1, \dots, l_q} \prod^q_{i =1} c_{l_i,t}(a(t)),$$
    where $C_{l_1, \dots, l_q}$ is a constant and $q$ is a positive integer not greater than $d$ and $l_1, \dots, l_q$ are also positive integers such that 
    $$ d - q -\sum^q_{i = 1} l_i =d -( n+1).$$
    Note that 
    $$ \deg_t\left(C_{l_1, \dots, l_q}\prod^q_{i =1} c_{l_i,t}(a(t))\right) \leq k_3 (\sum^{q}_{i = 1} (l_i + 1)) + k_1 \sum^q_{i = 1} l_i \leq k_3(n+1) + nk_1.$$

    Putting these all together, we have 
    $$ dc_{n,t}(a(t)) = P-Q$$
    which implies $c_{n,t}(x)$ is a polynomial in $\overline{K}[t][x]$ and also 
    $$\deg_t(c_{n,t}(a(t))) \leq k_3(n+1) + n k_1.$$
    
\end{proof}

\begin{lem}\label{lem: from-Bot-Cor-to-Analytic-relation}
   We retain the notation from Lemma~\ref{lem: linear-bound-c-deg}. Denote
   $$ m_1 \coloneq \lcm(\deg(b(t)), k_3)/\deg(b(t)).$$
   Suppose 

\[
\deg(b(t)) \geq \left(m_1 + 1\right)(k_3 + k_1).
\]
Then one has the expansion
\[
\left(
  b(t) + \sum_{i=1}^{\infty} c_{i,t}(a(t))\, b(t)^{-i}
\right)^{m_1}
= P_t\!\bigl(a(t), b(t)\bigr) + o(1),
\]
where $P_t(x,y) \in \overline{K}[t][x,y]$ and the notation $o(1)$ denotes terms whose degrees in $t$ are strictly negative.

\end{lem}
\begin{proof}
    Since $c_{j,t}(a(t)) \in \overline{K}[t]$ for all $j \in \N^+$ by Lemma \ref{lem: linear-bound-c-deg}, we have that any term of the form
    $$ \left(\prod^r_{j = 1} c_{k_j,t}(a(t)) \right)b(t)^{m_1 - r - \sum^r_{j = 1} k_j} $$
    is a polynomial in $a(t)$ and $b(t)$
    when $m_1 - r - \sum^r_{j = 1} k_j \geq 0$. From our assumption and Lemma \ref{lem: linear-bound-c-deg}, we have that 
    \begin{align}
        \deg(b(t)^j) \geq j(m_1 + 1)(k_3 + k_1) \geq 2j(k_3 + k_1) \nonumber \\> (j+1)k_3 + jk_1 \geq \deg_t(c_{j,t}(a(t))) 
    \end{align}
    for all $j \in \N^+$.

    Therefore, we only need to verify that 
    $$\deg\left(\left(\prod^r_{j = 1} c_{k_j,t}(a(t))\right)b(t)^{m_1 - r - \sum^r_{j = 1} k_j} \right) < 0$$
    for all choices of $r \in \{1,2, \dots, m_1-1\}$ and $\{k_1, \dots, k_r\}$ such that $$m_1 -r-\sum^r_{j= 1}k_j  < 0.$$
    
    Notice that, denoting $L = \sum^r_{j = 1}k_j$, we have
    $$\deg\left(\prod^r_{i = 1} c_{k_j,t}(a(t))\right) \leq (r + L)k_3 + Lk_1$$
    and 
    \begin{align}
        (L -m_1 +r)\deg(b(t)) \geq (L -m_1 +r) (m_1 + 1)(k_3 + k_1) \\
         > (L+r)k_3 + L k_1 \geq  \deg\left(\prod^r_{i = 1} c_{k_j,t}(a(t))\right)\nonumber,
    \end{align}
    since $(m_1 + 1)(k_3 + k_1) > (m_1 + 1) k_3 + (m_1 +1-r)k_1 $ and the coefficient of the $L$ term in 
    $$ (L -m_1 +r) (m_1 + 1)(k_3 + k_1)$$
    is $(m_1 + 1)(k_3 + k_1)$ which is strictly larger than the coefficient of the $L$ term in $$ (L+r)k_3 + L k_1,$$ which is $k_3 + k_1$.
     
    Therefore, we verified that 
    $$\deg\left(\left(\prod^r_{j = 1} c_{k_j,t}(a(t))\right)b(t)^{m_1 - r - \sum^r_{j = 1} k_j} \right) < 0.$$
    Thus, let $P_t(x,y) \in \overline{K}[t][x,y]$ be the polynomial part of 
    $$ \left(y + \sum^\infty_{i = 1} c_{i,t}(x)y^{-i}\right)^{m_1},$$
    we have 
    $$ \left(b(t) + \sum^\infty_{i = 1} c_{i, t}(a(t))b(t)^{-i}\right)^{m_1} = P_t(a(t),b(t)) + o(1).$$
\end{proof}
\begin{lem}\label{lem: xi-root-of-unity}
    We retain the notation and assumptions from Lemma \ref{lem: from-Bot-Cor-to-Analytic-relation} . Let $$(a_n(t), b_n(t)) \coloneq F_t^n(a(t), b(t)).$$ Suppose there exists a $\xi \in \C^*$ such that
    \begin{equation} \label{eq: anaytic-root-of-unity-1}
        \xi^{d^n}\left(a_n(t) + \sum^\infty_{i = 1} c_{i} a_n(t)^{-i}\right)^{m_2} = \left(b_n(t) + \sum^{\infty}_{i = 1} c_{i,t}(a_n(t))b_n(t)^{-i}\right)^{m_1}
    \end{equation}  holds for infinitely many $n \in \N$, where $c_i \in \overline{K}$ for each $i \in \N^+$. Then $\xi$ is a root of unity.
\end{lem}
\begin{proof}
    We first compare the leading coefficients in the both sides of the equations. We denote 
    $$ a(t) = \mu_a t^{a} + o(t^a)$$
    $$b(t) = \mu_bt^b + o(t^b).$$
    
    Note that $a > \deg_t(f_t)$ by our assumption. The left hand side of Equation (\ref{eq: anaytic-root-of-unity-1}) is equal to 
    $$ \xi^{d^n} \mu^{m_2d^n}_a t^{m_2ad^n} + o(t^{m_2ad^n}),$$
    and the right hand side is 
    $$ \mu_b^{m_1d^n}t^{m_1bd^n} + o(t^{m_1bd^n}).$$
    These imply that 
    $$
      am_2 = m_1b,  
    $$ 
\begin{equation}\label{eq: a-b-relation-1}
        (\mu_b^{m_1}/\mu_a^{m_2})^{d^n} = \xi^{d^n}.
    \end{equation}

    Now, we look at the coefficient of the first negative $t$-degree term. When $n$ is large enough, the first negative $t$-degree term of the left hand side of Equation (\ref{eq: anaytic-root-of-unity-1}) is a summation of terms of the form
$$ \xi^{d^n}a_n(t)^{-u_1}\prod^r_{ i = 1} c_{s_i},$$
where $u_1 \in \N^+$, $s_i, r$ are some positive integers bounded by $m_2$. Written compactly, the leading term is given by
    
    $$ \xi^{d^n} D_{F,1} \mu_a^{-u_1d^n} t^{- au_1d^n} ,$$
    where $D_{F,1}$ is a non-zero constant only depending on $F$, $a(t)$ and $m_2$. 
    
    While, similarly, the argument in Lemma \ref{lem: from-Bot-Cor-to-Analytic-relation} implies that
    the terms of the form $$\left(\prod^r_{j = 1} c_{k_j,t}(a_n(t)) \right)b_n(t)^{m_1 - r - \sum^r_{j = 1} k_j},$$
    for some $r \in \{1,2, \dots, m_1-1\}$ and $k_j \in \N^+$ such that $m_1 - r - \sum^r_{j = 1}k_j \geq 0$, are polynomials in $a(t)$ and $b(t)$ and hence are polynomials in $t$. Thus,
    the first negative $t$-degree term on the right hand side of Equation (\ref{eq: anaytic-root-of-unity-1}) is a summation of terms of the form 
    $$ b_n(t)^{-u_2}\prod^{r'}_{i = 1}c_{s'_i,t}(a_n(t))$$
    where $u_2 \in \N^+$, $r' \in \{1,2, \dots, m_1 -1\}, s_i' \in \N^+$ such that $$m_1 -r' -\sum^{r'}_{i = 1}s'_i = -u_2.$$ Then, written compactly, we again have that the leading term is given by
    $$ D_{F,2} \mu_a^{s d^n} \mu_b^{-u_2d^n} t^{asd^n +  C_2 - u_2bd^n},$$
    where $s$ is a positive integer, $D_{F,2}$ and $C_2$ are constant only depending on $F_t$, $a(t)$, $b(t)$ and $m_1$. Moreover, $D_{F,2}$ is also non-zero. Therefore, for every large enough $n \in \N^+$, the Equation (\ref{eq: anaytic-root-of-unity-1}) implies that 
    $$ a(s + u_1) = u_2b$$
    and 
    $$ \mu_a^{(s + u_1)d^n}/\mu^{u_2d^n}_b = \xi^{d^n}D_{F,1}/D_{F,2}.$$
    Plug in the relation of $b$ and $a$ to Equation \ref{eq: a-b-relation-1}, we also have 
    $$ \left(\mu_b^{m_2u_2/(s+ u_1)}/\mu_a^{m_2}\right)^{d^n} = \xi^{d^n}.$$
    Let $\nu = \mu_a^{s + u_1}/\mu_b^{u_2}$. These imply that 
    $$ \nu^{- m_2d^n} = \xi^{d^n(s + u_1)},$$
    $$ \nu^{d^n} = \xi^{d^n}D_{F_1}/D_{F_2}.$$
    Hence
    $$ \nu^{(s + u_1 + m_2)d^n} = (D_{F,1}/D_{F,2})^{s + u_1} ,$$
    for infinitely many $n \in \N$
    which implies that 
   $v$ is a root of unity
    and so $\xi$ is a root of unity.
\end{proof}

\subsection{Proof of Theorem \ref{thm:mainpolyskewinsubvar}}

We will also assume that $P_t$ satisfies assumption $\ref{assump:initialmarkedpoint}$ throughout the discussion as otherwise the theorem is trivial. If the first coordinate of $P_t$ is passive (not active i.e. not satisfying assumption \ref{assump:initialmarkedpoint}), then either $(f_t,a(t))$ is isotrivial or $a(t)$ is preperiodic under $f_t$. Since our assumption implies that there exists a $t_0 \in \overline{K}$ such that $a(t_0) \in \Prep(f_{t_0})$, we have $a(t) \in \Prep(f_t)$ in any case. Hence, it is trivial that $\overline{\Orb_{F_t}(P_t)}$ is proper Zariski closed.

 Recall that the $v$-adic Green function $G_{1,v}(t)$ (resp. $G_{2,v}(t)$) associated to $a(t)$ (resp. $P_t$) of degree $d\geq2$ is given by
\begin{equation}
	G_{1,v}(t)=\lim_{n\rightarrow\infty}\frac{1}{d^n}\log^+|f^n_t(a(t))|_v\label{eq:Greenoff}
\end{equation}
and 
\begin{equation}
	G_{2,v}(t)=\lim_{n\rightarrow\infty}\frac{1}{d^n}\log^+\|F_t^n(P_t)\|_v,\label{eq:Greenofpolyskew}
\end{equation}
respectively. We denote by $\mu_{f_t,v}=dd^c_t G_{1,v}(t)$ and $\mu_{F_t,v}=dd^c_t G_{2,v}(t)$ the bifurcation measures associated to $(f_t,a(t))$ and $(F_t,P_t)$, respectively. As usual, we write $$\hat{h}_{f_t,a}(t):=\hat{h}_{f_t}(a(t))=\frac{1}{[K(t):K]}\sum_{t'\in\text{Gal}(\overline{K}/K)\cdot t}\sum_{v\in M_K}N_vG_{1,v}(t').$$ 

\begin{proof}[Proof of Theorem \ref{thm:mainpolyskewinsubvar}] 
Suppose that there exists an infinite distinct sequence $\{t_n\}_{n\geq1}\subset \overline{K}$ such that $P_{t_n}$ is preperiodic under the action of $F_{t_n}(x,y)$. As explained in subsection \ref{sec:unlikelyinter}, we have that the sequence $\{t_n\}$ is equidistributed with respect to $\mu_{F_t,v}$ for all $v\in M_K$. From our assumption, we also have that such sequence of parameters $\{t_n\}$ satisfies that $a(t_n)$ is preperiodic under the action of $f_t(x)$. Applying the arithmetic equidistribution of small point \cite[Proposition 5.1]{BD13} to deduce that the sequence $\{t_n\}$ is equidistributed with respect to $\mu_{f_t,v}$ at all places $v$ of $K$. Thus it immediately yields the equality of measures $\mu_{F_t,v}=\mu_{f_t,v}$ at all $v\in M_K.$\\

Again, we view $K(t)$ as embedded in $\C$ with $t$ having large enough norm and trancedental over $K$. Then, by \cite[Proposition 2.6]{Jo99}, there exists a $R \in \R^*$ such that 
    in the region
    $$ D_2 \coloneq\{(x,y) \in \C^2 : G_{F_t}(x,y) - G_{f_t}(x,y) > R\}$$
    one has the vertical B\"ottcher coordinate $\Phi_x$ satisfying
   \begin{enumerate}
       \item $\Phi_x (y) = y + o(1)$ as $|y| \to \infty$;
       \item $\log |\Phi_x| = G_{F_t}(x,y) - G_{f_t}(x,y)$;
       \item 
       $\Phi_{f_t(x)}(g_t(x,y)) = \Phi_x(y)^d$.
   \end{enumerate}

   Since we have $\deg(b(t)) > \deg(a(t)) + \deg_t(g_t)$, denoting 
  
   $$(a_m(t),  b_m(t)) \coloneq F^m_t(a(t), b(t)),$$
   $$ l_1 \coloneq \deg(a(t)),$$
   $$ l_2 \coloneq \deg(b(t)),$$
   we have that
   $$\deg(a_m(t)) = l_1d^m$$
   $$ \deg(b_m(t)) = l_2d^m$$
   for any $m \in \N$. Then,
   $$ G_{f_t}(a_m(t)) = l_1d^m\log |t| + o(d^m\log|t|),$$
   $$ $$
   $$ G_{F_t}(a_m(t), b_m(t)) = l_2d^m\log|t|+ o(d^m\log|t|),$$
   as $m \to \infty$.
   Therefore, there exists a $N \in \N$ such that for any $m > N$, we have 
   $$(G_{F_t} - G_{f_t})(a_m(t), b_m(t)) > R$$
   and so
   $$ (a_m(t), b_m(t)) \in D_2.$$

   Let $D_1$ denote the region where the B\"ottcher coordinate of $f_t(x)$ is defined. By further enlarging $m$ if necessary we also have that 
   $$a_m(t) \in D_1.$$

   We abuse notation to let $(a(t), b(t))$ denote $(a_m(t), b_m(t))$ and recall that 
   $$ G_1(t) = G_{f_t}(a(t)) ,$$
   $$ G_2(t) =  G_{F_t}(a(t), b(t)),$$
   where we omit the place $v$ in the index and understand this as working over a fixed Archimedean place.
   Denote 
   $$ m_1 \coloneq \lcm(\deg(b(t)), \deg(a(t)))/\deg(b(t)),$$
   $$ m_2 \coloneq \lcm(\deg(b(t)), \deg(a(t)))/\deg(a(t)) - m_1.$$
   In particular, they satisfy that 
   $$ m_1 \deg(b(t)) = (m_1 + m_2) \deg(a(t))$$
   and 
   $$ m_1l_2 = (m_1 + m_2) l_1.$$

\begin{lem} \label{lem:equalGreen} $(m_2+ m_1)G_1=m_1G_{2}.$ 
\end{lem}

 \begin{proof} We first denote 
     \begin{align*} M_1 &= \{t \in \C : a(t) \text{ has bounded orbit under } f_t\},\\
      M_2 &= \{t \in \C : P_t \text{ has bounded orbit under } F_t\}.\end{align*}

      Note that, the above equidistribution argument shows that $\mu_{f_t,v} = \mu_{F_t, v}$ with respect to our fixed Archimedean place $v$. Thus,
      $$ \partial M_{1} = \Supp (\mu_{f_t,v}) = \Supp(\mu_{F_t, v}) = \partial M_2.$$

    Now, since the complements of both $M_1$ and $M_2$ are connected, by the maximum principle of harmonic functions, we have $M \coloneq M_1 = M_2$. Now, since $G_1, G_2 > 0$ in the complement of $M$ and the growth rate of $G_1$ and $G_2$ are $l_1d^m \log|t|$ and $l_2d^m \log|t|$ respectively as $|t| \to \infty$ where $m \in \N$, we have that $G_1$ and $G_2$ are Green's function for $M$ up to multiplications similar as in \cite[Remark 2.3]{BD13}. Therefore, $G_1$ and $G_2$ are only differed by a multiplications. 

   Notice that our assumption on $m_1$ and $m_2$ implies that $$(m_1 + m_2)l_1 = m_1 l_2.$$ Therefore, by comparing the growth rate of $G_1$ and $G_2$, we have 
   $$ (m_1 + m_2)G_1 = m_1 G_2$$

\end{proof}

   Now, note that $m_1(G_2 - G_1) = m_2 G_1$,
   together with the property (2) of the B\"ottcher coordinate, we have that 
   $$ |\Phi_{a(t)}(b(t))|^{m_1} = |\Phi(a(t))|^{m_2} $$
   for all $m \in \N$.

   This implies that there exists a $\xi \in \C^*$ such that 
   $$ \xi \Phi_{a(t)} (b(t))^{m_1} = \Phi(a(t))^{m_2}.$$
   Since $$\Phi_{a_{m}(t)}(b_{m}(t)) =\Phi_{a(t)}(b(t))^{d^m}$$
   $$ \Phi (a_{m}(t)) = \Phi(a(t))^{d^m},$$
   for any $m \in \N$, we have
   $$ \xi^{d^{m }} \Phi_{a_m(t)}(b_m(t))^{m_1} = \Phi(a_m(t))^{m_2},$$
   for all $m \in \N$

   Then by Lemma \ref{lem: xi-root-of-unity}, we have $\xi$ is a root of unity. Thus, there exists a root of unity $\xi'$, a positive integer $k$ and a natural number $i$ such that
   $$ \xi' = \xi^{d^{nk + i}}$$
   for all $n \in \N$. Therefore, 
   \begin{equation}\label{eq: skew-alge-relation-oribt}
       \xi' \Phi_{a_{ nk + i}(t)}(b_{nk + i}(t))^{m_1} = \Phi (a_{nk+i}(t))^{m_2},
   \end{equation} 
   for all $n \in \N$. 

   Now, by Lemma \ref{lem: from-Bot-Cor-to-Analytic-relation}, we have that 
   $$ \Phi_{a_{nk+i}(t)}(b_{nk + i}(t))^{m_1} = P(a_{ nk + i}(t) , b_{ nk + i}(t)) + o(1)$$
for some $P \in \overline{K}[t][x,y]$. Similarly, by \cite[Section 5.6]{BD13}, we have 
$$ \Phi(a_{ nk + i}(t))^{m_2} = Q(a_{ nk + i}(t)) + o(1),$$
where $Q \in \overline{K}[t][x]$. Therefore, the Equation \ref{eq: skew-alge-relation-oribt} implies that
$$ \xi'P(a_{nk +i}(t), b_{ nk +i}(t) ) + o(1) = Q(a_{nk + i}(t)) + o(1)$$
viewing as power series in $t$. This implies that 
$$ \xi'P(a_{ nk +i}(t), b_{ nk +i}(t) )  = Q(a_{ nk + i}(t)) ,$$
and then we have 
$$ F_t^{ nk + i}(a(t), b(t)) \in V(\xi'P(x,y) - Q(x))$$
for all $n \in \N$. This concludes the proof.

\end{proof}
\subsection{Implication on Conjecture \ref{conj: DM24-conj-1,1}}

In this subsection, we prove Corollary \ref{cor: main-skew-imply-conj}, which demonstrates the implication of Theorem \ref{thm:mainpolyskewinsubvar} towards Conjecture \ref{conj: DM24-conj-1,1}.

\begin{proof}[Proof of Corollary \ref{cor: main-skew-imply-conj}]
    Since Theorem \ref{thm:mainpolyskewinsubvar} implies that $\overline{\Orb_{\mathbf{\Phi}}(\mathbf{X})}$ is a curve in $\P^2_{\overline{K}(t)}$ and $\mathbf{X} $ is not preperiodic under $\mathbf{\Phi}$, we have $r_{\Phi, \mathcal{X}} = 1$. Thus, it is sufficient to show that 
    $$ \hat{T}\wedge \mathcal{X} \neq 0.$$ 

    Let $U$ be the affine chart where $\Phi|_{S \times U} (s,x,y)  =(s, F_s(x,y))$ for some regular polynomial skew products $F_s(x,y) = (f_s(x), g_s(x,y))$ of degree $d > 1$, where $f_s \in K[s][x]$ and $g_s \in K[s][x,y]$.
    Also, we denote 
    $$ F^n_s(x,y) = (f_{s,n}(x), g_{s,n}(x,y)),$$
     where $f_{s,n} \in K[x]$ and $g_{s,n} \in K[s][x,y]$.
    Let
    $$G(s,x,y) \coloneq \lim_{n \to \infty } d^{-n}\log\max\{1, |f_{s,n}(x)|, |g_{s,n}(x,y)|\}.$$
    Then 
    $$ (\hat{T}\wedge \mathcal{X} )|_{S \times U}= dd^c G(s,a(s),b(s)).$$
    Note that $s \to G(s,a(s), b(s))$ is subharmonic, non-constant and bounded from below as $\mathbf{X}$ is not preperiodic under $\mathbf{\Phi}$ and there exists $s_0 \in \overline{K}$ such that $\mathcal{X}_{s_0}$ is preperiodic under $F_{s_0}$, where $F_{s_0}$ denotes the regular polynomial skew products with $s = s_0$ plugged in. Hence $$dd^c G(s,a(s),b(s)) \neq 0.$$
\end{proof}

    \section*{ Acknowledgements}
    We would like to thank Prof. Jason Bell for helpful discussions. X.Z. was supported by the Natural Sciences and Engineering Research Council of Canada through a Discovery Grant (RGPIN-2022-02951).
	
\end{document}